\documentclass[12pt,reqno]{amsart} 
\usepackage[pagewise]{lineno}
\usepackage{amsaddr,amsmath} 

\usepackage{mathrsfs,dsfont} 
\usepackage{wrapfig} 
\usepackage{graphicx} 
\usepackage{tkz-berge}
\usepackage{tikz-cd}
\usepackage{changepage}
	\usetikzlibrary{positioning}
	\usepackage[linktoc=all]{hyperref} \hypersetup{ colorlinks, linkcolor=black, urlcolor= black } 
	\usepackage{enumerate}
\allowdisplaybreaks 
\usepackage{tikz}	
\usepackage{geometry}  

\usepackage{amssymb}

	\theoremstyle{plain}
	\newtheorem{theorem}{Theorem}[section]
	\newtheorem{lem}[theorem]{Lemma}
	\newtheorem{prop}[theorem]{Proposition}
	\newtheorem{cor}[theorem]{Corollary}
	\theoremstyle{definition}
	\newtheorem{defn}[theorem]{Definition}
	\newtheorem*{acknowledgements}{Acknowledgements}
	\theoremstyle{remark}
	\newtheorem{rmk}[theorem]{Remark}
	\newtheorem{eg}{Example}

	\newcommand{\C}{\mathbb{C}}
	\newcommand{\X}{\mathcal{X}} 
	\newcommand{\N}{\mathbb{N}}
	\newcommand{\Z}{\mathbb{Z}}
	
	\newcommand{\E}{\mathcal{E}}

	\newcommand{\Q}{\mathbb{Q}}

	\renewcommand{\d}{\text{d}} 
	\newcommand{\<}{\langle}
	\renewcommand{\>}{\rangle}
	\newcommand{\sbst}{\subseteq}
	\renewcommand{\l}{\ell}
	\newcommand{\Mod}[1]{\ \mathrm{mod}\ #1}
	\renewcommand{\L}{\mathscr{L}}
	\renewcommand{\SS}[3]{(#1,#2,#3)_{p}}
	\newcommand{\ord}{\text{ord}_p}

\hypersetup{citecolor=olive }

\begin{document}

\newgeometry{margin=1in}

\title[A symmetric $p$-adic symbol for triples of modular forms]
{A symmetric $p$-adic symbol for triples of modular forms} 

\author[Wissam Ghantous]{Wissam Ghantous}
\address{University of Central Florida, \\ 4393 Andromeda Loop N, Orlando, FL 32816, USA}
\subjclass[2010]{11F12, 11G05, 11G35, 11G40, 11Y40.}

\date{\today }

\keywords{$p$-adic triple symbol, $p$-adic Abel-Jacobi map, Poincaré pairing, nearly overconvergent projection, ordinary projection, slope projection.}

\begin{abstract}
	In \cite{DR14}, the authors define the Garrett-Rankin triple product $p$-adic $L$-function 
	and relate it to the image of certain diagonal cycles under the $p$-adic Abel-Jacobi map. 
	We introduce a new $p$-adic triple symbol based on this $p$-adic $L$-function and show that it satisfies symmetry relations, 
	when permuting the three input modular forms. 
	We also provide computational examples illustrating this symmetry property. 
	To do so, we extend the algorithm provided in \cite{Lau14} to allow for ordinary projections 
	of nearly overconvergent modular forms -- 
	not just overconvergent modular forms -- as well as certain projections over spaces of non-zero slope. 
	Our work also gives an efficient method to calculate certain Poincaré pairings in higher weight, 
	which may be of independent interest. 
\end{abstract}

\maketitle 

\tableofcontents

\section{Introduction}  \label{IntroOfPaper}
	
	Let $f,g,h$ be three normalised cuspidal modular eigenforms (for all the Hecke operators) 
	over $\Q$ of weight $2$, level $N$ and trivial characters. 
	Fix a prime $p \ge 5$ and assume that $p \nmid N$. 
	Let $\alpha_{f},\beta_{f}$ be the roots of the \emph{Hecke polynomial} 
		\begin{equation*} 
			x^{2} - a_p(f) x + p. 
		\end{equation*} 
	Then the modular form $f$ is \emph{regular at $p$}, i.e. the roots $\alpha_f$ and $\beta_f$ are different (cf. \cite{CE98}). 
	Assume that $f$ is \emph{ordinary at $p$}, i.e. that one of the roots of $x^{2} - a_p(f) x + p$, say $\alpha_{f}$, 
	is a $p$-adic unit. 
	Define the following two modular forms: 
		\begin{align} \label{22Sep210442p}
		\begin{split} 
			f_\alpha (q)&:= f(q)- \beta_f f(q^{p}); \\ 
			f_\beta (q)&:= f(q)- \alpha_f f(q^{p}).
		\end{split} 
		\end{align} 
	We call $f_\alpha$ and $f_\beta$ the \emph{$p$-stabilizations} of $f$. 
	They have level $pN$, and are eigenforms for the $U_p$ operator with respective eigenvalues $\alpha_f$ and $\beta_f$. 
	Since we assumed that $\alpha_{f}$ is a unit, 
	it is customary to call $f_\alpha$ the \emph{ordinary $p$-stabilization} of $f$. 
	Define the following Euler factors: 
	\begin{align} \label{22Oct111106a}
	\begin{split}
	\begin{array}{ll}
			\mathcal{E}(f,g,h):= (1-\beta_f \alpha_g \alpha_h p^{-2}) (1-\beta_f \alpha_g \beta_h p^{-2}) (1-\beta_f \beta_g \alpha_h p^{-2}) (1-\beta_f \beta_g \beta_h p^{-2}); \\ 
			\tilde{\mathcal{E}}(f,g,h):= (1-\alpha_f \alpha_g \alpha_h p^{-2}) (1-\alpha_f \alpha_g \beta_h p^{-2}) (1-\alpha_f \beta_g \alpha_h p^{-2}) (1-\alpha_f \beta_g \beta_h p^{-2}) ; \\
	\end{array} 
	\\
	\begin{array}{llllllllll} 
		\mathcal{E}_0(f) := 1-\beta_f^{2} \, p^{-1}; & & \tilde{\mathcal{E}_0}(f) := 1-\alpha_f^{2} \, p^{-1}; & & & & & & & \\ 
		\mathcal{E}_1(f) := 1-\beta_f^{2} \, p^{-2}; & & \tilde{\mathcal{E}_1}(f) := 1-\alpha_f^{2} \, p^{-2}. & & & & & & & \\
	\end{array}
	\end{split}
	\end{align}

	Assume $f$ is a newform and let $\lambda_{f_{\alpha}}$ be the 
	unique Hecke-equivariant linear functional on $S_2(\Gamma_0(N))$ 
	that factors through the Hecke eigenspace associated to $f_{\alpha}$
	and that is normalised to send $f_{\alpha}$ to $1$ (cf. Definition 2.7 in \cite{Loe18}). 
	Assume as well that the $f_{\beta}$-eigenspace is semisimple 
	and define $\lambda_{f_{\beta}}$ analogously to $\lambda_{f_{\alpha}}$. 
	Let $\text{d}:= q \frac{\d}{\d q}$ be the Serre differential operator 
	and $\omega_f:= f(q) \frac{\text{d}q}{q}$ the differential associated to $f$. 
	Consider the quantity 
		\begin{equation} \label{22Sep211134a}
			\frac{\left\<\omega_f, \phi(\omega_f) \right\>}{p} 
			\left( \frac{\mathcal{E}_1(f)}{\mathcal{E}(f,g,h)} \ \beta_{f} \ \lambda_{f_{\alpha}} \left( \text{d}^{-1}(g^{[p]} ) \times {h} \right) 
			+\frac{\tilde{\mathcal{E}}_1(f)}{\tilde{\mathcal{E}}(f,g,h)} \ \alpha_{f} \ \lambda_{f_{\beta}} 
			\left( \text{d}^{-1}(g^{[p]} ) \times {h} \right) \right), 
		\end{equation} 
	where $\< \cdot , \cdot \>$ is the Poincaré pairing (cf. Theorem 5.2 in \cite{Col95}) and $\phi$ is the Frobenius map. 
	It turns out 
	that this quantity is independent -- up to a sign -- 
	of the order of $f,g$ and $h$. 
	This result is particularly surprising since the quantity in (\ref{22Sep211134a}) 
	does not appear to be symbolically symmetric in $f,g$ and $h$. 
	This will fit into the framework of our paper, as we relate this quantity to 
	the image of certain diagonal cycles under the $p$-adic Abel-Jacobi map. 
	
	The above can even be generalised to modular forms of higher weight and any characters satisfying $\chi_f \chi_g \chi_h=1$, 
	which we will do in Section \ref{NewLfunction}. 
	In that case, one needs to adjust the Euler factors from (\ref{22Oct111106a}) 
	and introduce an extra factor and some twists by $\chi_f^{-1}$ in (\ref{22Sep211134a}). 
	One would also require that the weights be balanced, i.e. 
	that the largest one is {strictly} smaller than the sum of the other two. 
	\\ 
	
	In order to explicitly calculate (\ref{22Sep211134a}), for modular forms of general weight, we need certain computational tools, 
	namely being able to compute ordinary projections of nearly overconvergent modular forms, 
	as well as projection over the slope $\sigma$ subspace for $\sigma$ not necessarily zero. 
	In \cite{Lau14} (see also \cite{Lau11}), the author describes an algorithm allowing the calculation of ordinary projections of overconvergent modular forms. 
	We introduce here improvements to this algorithm, allowing us to accomplish the aforementioned tasks. 
	We illustrate this in the experimental calculations detailed in Section \ref{22Oct260529p}. 
	The use of this new algorithm is not restricted to this paper. 
	\\ 
	
	An additional application of our code is the calculation of certain periods of modular forms. 
	Indeed, using the symmetry of our new $p$-adic triple symbol, introduced in Section \ref{22Oct210251p}, 
	we explain how one can use our algorithms to compute 
	the Poincaré pairing $\Omega_f:= \< \omega_f, \phi(\omega_f) \>$, where $\phi$ denotes the Frobenius action and $f$ 
	is a newform of any weight over $\Q$. 
	See \cite{DL21}, \cite{DLR16} and Section III.5 of \cite{Nik11} for instances where this pairing appears in the literature. 
	There are currently no known ways of evaluating general Poincaré pairings, and the value of $\Omega_f$ 
	has so far only been computed in cases where $f$ has weight $2$ using Kedlaya's algorithm \cite{Ked01}. 
	
	$ $
	\\ 
	\textbf{Roadmap.} Our paper is structured as follows. 
	In Section \ref{SectionBackground}, we introduce the main theoretical notions used in this work. 
	That is, the main results concerning overconvergent and nearly overconvergent modular forms, 
	the Katz basis, the $U_p$ operator, and finally Lauder's algorithm introduced in \cite{Lau14} 
	to compute ordinary projections of overconvergent modular forms. 
	In Section \ref{ExplicitAlgorithmicMethods}, we describe improvements to that algorithm. 
	Firstly, we introduce an overconvergent projector, 
	and use it for the calculation of the ordinary projection of a nearly overconvergent modular form. 
	Secondly, we describe a method to compute the projection of a nearly overconvergent modular form 
	over the slope $\sigma$ subspace, 
	where $\sigma$ is not necessarily zero. 
	Section \ref{NewLfunction} is dedicated to our new $p$-adic triple symbol. 
	We first describe the Garrett-Rankin triple product $p$-adic $L$-function defined in \cite{DR14} and its relation to 
	the $p$-adic Abel-Jacobi map. 
	We then use this relation to introduce a new 
	$p$-adic triple symbol $\SS{f}{g}{h}$. 
	After that, we focus on studying the symmetry properties of $\SS{f}{g}{h}$ 
	when permuting ${f,g}$ and ${h}$. 
	Finally, in Section \ref{22Oct260529p}, we provide some computational examples 
	and show how to compute the Poincaré pairing $\< \omega_f, \phi(\omega_f) \>$, 
	where $\phi$ denotes the Frobenius action and $f$ 
	is a newform with rational coefficients of any weight.

\begin{acknowledgements}
	I would like to thank my DPhil supervisor Alan Lauder for his help and guidance over the past years. 
	I am also very grateful to Henri Darmon, Adrian Iovita, David Loeffler and Victor Rotger. 
	Their ideas, advice and input to this article 
	were extremely helpful. 
	I would like to specifically thank David Loeffler for his ideas on overconvergent projections and slope projections 
	which are central to Section \ref{ExplicitAlgorithmicMethods}. 
	I also particularly thank Victor Rotger for his advice on Section \ref{22Oct250500p}. 
	Furthermore, I am deeply indebted to Henri Darmon for suggesting this project to me 
	and giving generous help throughout this endeavour. 
	Lastly, I would like to thank the anonymous reviewer for their helpful and detailed advice, 
	allowing me to improve many aspects of the paper.   
\end{acknowledgements}

\section{Background} \label{SectionBackground}

\subsection{Modular forms} \label{ModularForms} 
	Throughout this paper, we will mainly deal with overconvergent and nearly overconvergent modular forms 
	of integer weight, 
	a thorough account of which is given in \cite{Kat73} (for overconvergent forms) 
	and \cite{DR14} or \cite{Urb14} (for their nearly-overconvergent counterpart). 
	In this section, we will review their computational aspects as well as the necessary results that will be used in the rest of the paper. 

	Let  $B$ be 
	be the ring of integers of a finite extension $K$ of $\Q_p$ such that there exists $r \in B$ with $0< \ord(r) < {1}/{(p+1)}$.  
	Denote by $M_{k}^{p\text{-adic}}(B,\Gamma; r)$ the space of $p$-adic modular forms of weight $k$, level $\Gamma$ 
	and growth condition $r \in B$ (cf. Section 2.2 in \cite{Kat73}). 
	We might often drop the $B$ when there is no possible confusion. 
	In the case where $\ord(r) \not =0$, we say that we have an \emph{overconvergent modular form} 
	of weight $k$, level $\Gamma$ and growth condition $r$. 
	We denote the space of such overconvergent modular forms by $M_{k}^{\text{oc}}(B,\Gamma; r)$. 
	We also define the $p$-adic Banach space $M_{k}^{\text{oc}}(K,N; r) := M_{k}^{\text{oc}}(B,N; r) \otimes_{B} K$, 
	where the unit ball is given by $M_{k}^{\text{oc}}(B,N; r)$. 

	The usual definitions of overconvergent modular forms, viewing them as functions on test objects 
	or as sections of certain line bundles (cf. \cite{Gou88, Kat73}) are not very amenable to computations, as they are quite abstract. 
	An alternative way to work with such modular forms is through the Katz basis, 
	allowing us to express overconvergent modular forms as series in classical objects. 
	
	First, assume that $p\ge 5$ and does not divide $N$. 
	Let $E_{p-1}$ denote the normalised Eisenstein series of weight $p-1$ (and level $1$). 
	We write $M_{k}(B,N)$ to mean the space of modular forms over $B$ of weight $k$ 
	and level $\Gamma_1(N)$. 
	Note that we have $$M_{k}(B,N) = M_{k}(\Z_p,N) \otimes_{\Z_p} B.$$ 
	The map 
		\begin{align*}
			M_{k+(i-1)(p-1)}(B,N) & \hookrightarrow M_{k + i(p-1)}(B, N) \\ 
				 f & \mapsto E_{p-1} \cdot f
		\end{align*} 
	is injective but not surjective for all $i \ge1$. 
	It also has a finite free cokernel (\cite{Kat73}, Lemma 2.6.1), so it must split. 
	We can then, following Gouvêa’s notation in \cite{Gou88}, let $A_{k + i (p-1)}(B,N)$ be a free $B$-module
	such that 
		\[M_{k + i(p-1)}(B,N) = E_{p-1} \cdot M_{k+(i-1)(p-1)}(B,N) \oplus A_{k+i(p-1)}(B,N). \] 
	For $i=0$, let $A_{k}(B,N) := M_k(B,N)$. 
	We also have 
		\[A_{k+i(p-1)}(B,N)=A_{k+i(p-1)}(\Z_p,N) \otimes_{\Z_p} B. \] 
	We can think of $A_{k+i(p-1)}(B,N)$ as the set of modular forms of weight $k + i (p-1)$ that do not come from 
	smaller weight forms multiplied by $E_{p-1}$. 
	We notice that we can write 
		\begin{align*} 
				M_{k + i(p-1)}(B,N) 
				&= \bigoplus_{a=0}^{i} (E_{p-1})^{i-a} \cdot A_{k+a(p-1)}(B,N). 
		\end{align*} 
	We may now give an equivalent definition for the space of $r$-overconvergent modular forms 
	(cf. Proposition I.2.6 in \cite{Gou88}). 
	
	\begin{prop} \label{OCMFs3}
		The space of overconvergent modular forms of weight $k \ge 2$, growth condition $r \in B$, $\ord(r)\not =0$, and level $\Gamma_1(N)$ 
		is given by 
		\begin{align} \label{Nov20700p}
		M_{k}^{\text{oc}}(N; r)
			&= \left \{ \sum_{i=0}^{\infty} r^{i} \frac{b_i}{E_{p-1}^{i}} : b_i \in A_{k+i(p-1)}(B,N) , \lim_{i \to \infty} b_i=0 \right \}, 
		\end{align} 
		where by $\lim_{i \to \infty} b_i=0$, we mean that the $q$-expansion of $b_i$ is more and more divisible by $p$ as $i$ goes to infinity. 
	\end{prop}
	
	\begin{rmk} 
		Similarly, we can define the space of overconvergent modular forms with a given character $\chi$ to be 
			\begin{align*} 
			M_{k}^{\text{oc}}(N, \chi; r) 
				:= \left \{ \sum_{i=0}^{\infty} r^{i} \frac{b_i}{E_{p-1}^{i}} : b_i \in A_{k+i(p-1)}(B,N, \chi) , \lim_{i \to \infty} b_i=0 \right \}, 
		\end{align*} 
		where $A_{k+i(p-1)}(N, \chi)$ is defined analogously to $A_{k+i(p-1)}(N)$ by 
			\[M_{k + i(p-1)}(B,N,\chi) = E_{p-1} \cdot M_{k+(i-1)(p-1)}(B,N,\chi) \oplus A_{k+i(p-1)}(B,N, \chi). \] 
	\end{rmk} 
	 
	An expansion for $f \in M_{k}^{\text{oc}}(B,N,\chi; r)$ of the form $f=\sum_{i=0}^{\infty} r^{i} \frac{b_i}{E_{p-1}^{i}}$ 
	is called a \emph{Katz expansion}. 
	We can also talk about the space of all overconvergent modular forms $M_{k}^{\text{oc}}(B,N,\chi)$ 
	without specifying the growth condition, 
		\[M_{k}^{\text{oc}}(B,N,\chi):= \bigcup_{\substack{r \in B \\ \ord(r) >0}} M_{k}^{\text{oc}}(B,N,\chi; r). \] 
	Similarly, we let $M_{k}^{p\text{-adic}}(B,N,\chi)$ denote the space of all $p$-adic modular forms 
	of level $\Gamma_1(N)$, weight $k$, character $\chi$ and any growth condition $r \in B$. 
	If $r=r_0 r_1$, for $r_0,r_1 \in B$, we then have an inclusion (c.f. Corollary I.2.7 in \cite{Gou88}) 
		\begin{align} \label{dec181104p} 
			\begin{split}
				M_{k}^{p\text{-adic}}(B,N,\chi; r) & \hookrightarrow M_{k}^{p\text{-adic}}(B,N,\chi; r_0), \\ 
				 \sum_{i=0}^{\infty} r^{i} \frac{b_i}{E_{p-1}^{i}} & \mapsto \sum_{i=0}^{\infty} r_0^{i} \frac{(r_1^{i} b_i)}{E_{p-1}^{i}}.
			\end{split} 
		\end{align}
	We now introduce the Serre differential operator (cf. Theorem 5 in \cite{Ser73}) 
		\begin{align*} 
		\begin{split} 
		q \frac{\text{d}}{\text{d}q} : M_{k}^{p\text{-adic}}(B,\Gamma_1(N)) &  \longrightarrow M_{k+2}^{p\text{-adic}}(B,\Gamma_1(N)) \\ 
		\sum_{n} a_n q^{n} & \mapsto \sum_{n} n a_n q^{n}.
		\end{split}
		\end{align*} 
 
	This operator does not necessarily preserve overconvergence in general. 
	We do however, have the following special case. 
	\begin{theorem}[Proposition 4.3, \cite{Col95}] \label{21may251008a}
		Let $k \ge 1$ and $f \in M_{1-k}^{\emph{oc}}(K,\Gamma_1(N))$. 
		Then, $\left(q \frac{\text{d}}{\text{d}q} \right)^{k}f \in M_{1+k}^{\emph{oc}}(K,\Gamma_1(N))$. 
	\end{theorem}

	We now discuss nearly overconvergent modular forms. 
	Let $M_{k}^{\text{n-oc}}(K,\Gamma;r;s)$ denote the space of 
	nearly overconvergent modular forms of weight $k$, level $\Gamma$, growth condition $r \in B$ 
	and order of near overconvergence less or equal to $s \in \Z_{\ge 0}$, 
	as defined in Paragraph 3.2.1 of \cite{Urb14}. 
	When $s=0$, we retrieve the usual definition of overconvergent modular forms. 
	We have inclusions 
		\[M_{k}^{\text{oc}}(K,\Gamma_1(N);r) \sbst M_{k}^{\text{n-oc}}(K,\Gamma_1(N); r;s) 
					\sbst M_{k}^{p\text{-adic}}(K,\Gamma_1(N)) \] 
	for all $r$ and $s$. 
	We now give some concrete characterizations of nearly overconvergent modular forms 
	using the Eisenstein series $E_2$. Recall that $E_2$
	 is transcendental over the ring of overconvergent modular forms (cf. \cite{CGJ95}), so 
		\begin{equation} \label{Nov061122p}
			M_{k}^{\text{oc}}(K,\Gamma) (E_2) \cong M_{k}^{\text{oc}}(K,\Gamma)(x), 
		\end{equation} 
	where $x$ is a free variable. 
	
	\begin{prop}[Remark 3.2.2 in \cite{Urb14}] \label{Nov061020p}
		Let $f$ be a nearly overconvergent modular form over $K$ of weight $k$, 
		level $\Gamma_1(N)$ and order less or equal to $s$. 
		Then there exist overconvergent modular forms $g_0, g_1,...,g_s$ 
		with $g_i \in M_{k-2i}^{\emph{oc}}(K,\Gamma_1(N))$ such that 
			\begin{equation} \label{Nov061129p} 
				f = g_0 + g_1 E_2 + ... + g_s E_2^{s}. 
			\end{equation} 
	\end{prop} 
	By Proposition \ref{Nov061020p} and Equation (\ref{Nov061122p}), 
	nearly overconvergent modular forms are polynomials in $E_2$ with overconvergent modular forms as coefficients, 
	so we can view them as elements of $M_{k}^{\text{oc}}(K,\Gamma_1(N))(x)$. 
	Hence, on top of having a $q$-expansion in $K[[q]]$ 
	they also have a polynomial $q$-expansion in $K[[q]][x]$ 
	(of degree less or equal to $s$, where $s$ is the order of near overconvergence) 
	that comes from Equation (\ref{Nov061129p}). 
	Consider the operator $\delta_k$ acting on nearly overconvergent modular forms of weight $k$ 
	defined on polynomial $q$-expansions (cf. Section 3.2 of \cite{Urb14}) as 
		\begin{equation*} 
			(\delta_{k}f) (q, x):= q \frac{\text{d}}{\text{d}q}f + k x f(q). 
		\end{equation*} 
	Then $\delta_{k}$ sends modular forms of weight $k$ to modular forms of weight $k+2$. 
	Define as well the iterated derivative $\delta_{k}^{s} := \delta_{k+2s-2} \circ \delta_{k+2s -4} \circ ... \circ \delta_{k}$. 
	\begin{prop}[Lemma 3.3.4 in \cite{Urb14}] \label{Nov061021p} 
		Let $f$ be a nearly overconvergent modular form of weight $k$ and order less or equal to $s$ such that $k > 2s$. 
		Then, for each $i = 0, ..., s$, there exists a unique overconvergent modular form $h_i$ of weight $k - 2i$ such that 
			\[f = \sum_{i=0}^{s} \delta_{k-2i}^{i} (h_i). \] 
	\end{prop} 
	
	Propositions \ref{Nov061020p} and \ref{Nov061021p} allow us to think about 
	nearly overconvergent modular forms as having an overconvergent component in them. 
	We define the overconvergent projection 
	of nearly overconvergent modular forms as 
		\begin{equation*}
			\pi_{\text{oc}}\left( \sum_{i=0}^{s} \delta_{k-2i}^{i} (h_i) \right):=h_0. 
		\end{equation*} 
	
\subsection{The $U_p$ operator} \label{TheUpOperator} 
	Consider the Hecke, Atkin and Frobenius operators $T_p$, $U_p$ and $V$ acting on $p$-adic modular forms 
	via 
		\begin{align*}  
			T_p : &\sum_{n} a_n q^{n} \mapsto \sum_{n} a_{pn} q^{n} + \chi(p) \ p^{k-1} \sum_{n} a_{n} q^{pn}, \\
			U_p : &\sum_{n} a_n q^{n} \mapsto \sum_{n} a_{pn} q^{n}, \\ 
			V : &\sum_{n} a_n q^{n} \mapsto \sum_{n} a_n q^{pn}.
		\end{align*} 

	We notice that $V$ is a right inverse for $U_p$ and that $VU_p \left(\sum_{n} a_n q^{n} \right)=\sum_{p|n} a_n q^{n}$. 
	In particular, $U_p$ has no left inverse and we write, for a modular form $f:= \sum_{n} a_n q^{n}$, 
		\begin{equation*} 
			U_pV (f) = f, \qquad f^{[p]}:= (U_pV-VU_p)(f) = \sum_{p  \not \! \  | n} a_n q^{n}. 
		\end{equation*} 
	We call $f^{[p]}$ the $p$-depletion of $f$. 
	We have the formula $U_p(V(f) \cdot g) = f \cdot U_p(g)$, 
	for modular forms $f$ and $g$, which can be proven by looking at $q$-expansions. 
	In particular, this says that $U_p$ is multiplicative when one 
	of its inputs is in the image of the \emph{Frobenius} map $V$.

	If we restrict our attention to the case $0< \ord(r) < \frac{1}{p+1}$ (cf. Corollary II.3.7 in \cite{Gou88}, and the discussion following it), we have an inclusion 
		\begin{align*} 
				p \cdot U_p : M_{k}^{\text{oc}}(B,N; r) & \hookrightarrow M_{k}^{\text{oc}}(B,N; r^{p}). 
		\end{align*} 
	So $U_p \left( M_{k}^{\text{oc}}(K,N; r) \right) \sbst \frac{1}{p} M_{k}^{\text{oc}}(K,N; r^{p})$. 
	Combining this with the fact that $M_{k}^{\text{oc}}(B,N; r^{p})  \sbst M_{k}^{\text{oc}}(B,N; r)$ 
	via the map in (\ref{dec181104p}), 
	we can view the Atkin operator $U_p$ as an endomorphism of $M_{k}^{\text{oc}}(K,N; r)$, when $0< \ord(r) < \frac{p}{p+1}$. 
	This endomorphism 
	is completely continuous so we can apply $p$-adic spectral theory, as in \cite{Ser62}. 
	We therefore obtain that the Atkin operator $U_p$ will induce a decomposition (as in Section 2 of \cite{Wan98}), 
	for all $\sigma \in \Q_{\ge 0}$, 
	on the space of overconvergent modular forms: 
		\begin{equation} \label{23Jul120422p}
			M_k^{\text{oc}}(K,N; r) =  M_k^{\text{oc}}(K,N; r)^{\text{slope } \sigma}  
						\oplus X_{\sigma}, 
		\end{equation} 
	where $M_k^{\text{oc}}(K,N; r)^{\text{slope } \sigma}$ is the finite dimensional space of 
	overconvergent modular forms in $M_k^{\text{oc}}(K,N; r)$ of slope $\sigma$. 	
	Recall that the slope $\sigma$ subspace 
	is the generalised eigenspace of $U_p$ whose eigenvalues have $p$-adic valuation $\sigma$. 
	A similar decomposition to Equation (\ref{23Jul120422p}) also holds for classical modular forms. 
	
	The overconvergent modular forms of slope zero are said to be ordinary 
	and we denote this space by $M_k^{\text{oc,ord}}(N)$.  
	Hida’s ordinary projection operator $e_{\text{ord}} := \lim U_{p^{n!}}$ acts on overconvergent modular forms 
	and projects the entire space $M_k^{\text{oc}}(N)$ onto its subspace of ordinary forms $M_k^{\text{oc, ord}}(N)$. 
	Finally, Equation (\ref{23Jul120422p}) tells us that any overconvergent modular form $f$ 
	has a component $f_{\sigma}$ in each given slope $\sigma$. 
	This gives us a notion of slope projection $e_{\text{slope } \sigma}(f) := f_{\sigma}$, 
	with $e_{\text{slope } 0}=e_{\text{ord}}$. 
	For computational purposes, when considering a specific $\sigma$, 
	we will often make the assumption that the slope $\sigma$ space is semisimple. 
	This will be explicitly mentioned to avoid any confusion. 
	
	Consider now the space of nearly overconvergent modular forms. 
	One can still define the ordinary projection operator as $e_{\text{ord}} := \lim U_{p^{n!}}$, 
	since this operator is actually defined for any $p$-adic modular form in general. 
	It turns out that the ordinary projection of a nearly overconvergent modular form 
	only depends on its overconvergent part. 
	\begin{theorem}[Lemma 2.7 in \cite{DR14}] \label{Nov050333p} 
		Let $F$ be a nearly overconvergent modular form, 
		then 
			\[e_{\text{ord}}(f) = e_{\text{ord}} \pi_{\text{oc}}(f). \]
	\end{theorem} 
	 Thus, taking ordinary projections of \emph{nearly overconvergent} modular forms reduces to 
	 taking ordinary projections of \emph{overconvergent} modular forms. 
	\\ 
	
	As explained in Section 3.3.6 of \cite{Urb14} (see also Appendix II of \cite{AI21} for Urban's erratum to \cite{Urb14}), the Atkin operator $U_p$ is also completely continuous 
	when viewed as an endomorphism of $M_{k}^{\text{n-oc}}(K,N; r;s)$, 
	for $0< \ord(r) < \frac{1}{p+1}$. 
	Hence, we obtain a decomposition of $M_{k}^{\text{n-oc}}(K,N; r;s)$ similar to that of Equation (\ref{23Jul120422p}). 
	This means that we may also similarly speak of slope $\sigma$ projections $e_{\text{slope } \sigma}(f)$ 
	for nearly overconvergent modular forms $f$.

	\subsection{Ordinary projections of overconvergent modular forms} \label{OrdProjOCMFs} 
	As we are interested in performing explicit computations in this paper, 
	we will approximate our overconvergent modular forms in $M^{\text{oc}}_{k}(B, N, \chi ; r)$ 
	by truncated power series (i.e. polynomials) modulo $p^{m}$, in $\Z[[q]]/(q^{h}, p^{m})$ 
	for some $m \in \N$ and $h=h(m,N,k,\chi) \in \N$. 
	While one can attempt to keep track of the precision lost at various steps of our computations, 
	in order to obtain a final error bound for our calculations, 
	we do not give provably correct error bounds. 
	Instead, we use a simpler ad hoc approach to measure the precision of our outputs ($p$-adic numbers): 
	we simply run our algorithm multiple times, to different precisions, and see 
	by what power of $p$ they differ. 
	\\ 
	
	We first explain how to write down the Katz expansion of an overconvergent modular 
	form in $M^{\text{oc}}_{k}(B, N, \chi ; r)$, with $\ord(r)=\frac{1}{p+1}$, as well as a matrix representing $U_p$. 
	More details can be found in Section 3 of \cite{Lau11}. 
	Note that by the inclusion in (\ref{dec181104p}), the space $M^{\text{oc}}_{k}(B, N, \chi ; r)$ 
	contains modular forms of growth rate having valuation $\frac{p}{p+1}$. 
	Picking a row-reduced basis $B_i$ for each $A_{k+i(p-1)}(\Z_p, N)$, 
	we obtain the sets: 
		\begin{equation} \label{25jan141045p}
		\text{Kb}:= 
				\left \{ \frac{ r^{i} \cdot b}{E_{p-1}^{i}} \mod (q^{h’ p}, p^{m’}) : 
					b \in B_i, i = 1,..., \left \lfloor \frac{(p+1)m}{p-1} \right \rfloor\right \}, 
		\end{equation}
		\begin{equation}\label{25jan141046p}
		S := 
				\left \{ U_p \left(\frac{ r^{i} \cdot b}{E_{p-1}^{i}}\right) \mod (q^{h’}, p^{m’}) : 
					b \in B_i, i = 1,..., \left \lfloor \frac{(p+1)m}{p-1} \right \rfloor\right \}, 
		\end{equation}
	for an appropriate choice of $h'$ and $m'$ depending on $m$. 
	We call $\text{Kb}$ the \emph{Katz basis}. 
	Let $d$ be its size and write $\text{Kb}=\{v_1,...,v_d\}$. 
	Any overconvergent modular form of growth condition having valuation $\frac{1}{p+1}$, 
	when reduced modulo $(q^{h’ p}, p^{m’})$, can be expressed as a linear combination in $\text{Kb}$. 
	
	Let $E$ and $T$ be the $d \times h’$ matrices formed by 
	taking the elements of Kb and $S$ respectively and looking at the first $h’$ terms in their $q$-expansions.  
	Compute the $d \times d$ matrix $A’$ such that $T= A’ E$. 
	Then, $A:= A’ \mod p^{m}$ is the representation of the operator $U_p$ in the Katz basis. 
	We write $A=[U_p]_{\text{Kb}}$. 
	\begin{rmk}\label{25jan141123p}
	In practice (cf. Algorithm 1 of \cite{Lau11}), 
	we may perform a simplification to the above algorithm. 
	Indeed, Lauder explains in Section 3.2.1 of \emph{op. cit.} that we may replace $\text{Kb}$ and $S$, 
	in (\ref{25jan141045p}) and (\ref{25jan141046p}), by 
		\[\text{Kb}:= 
				\left \{ \frac{ p^{\left \lfloor \frac{i}{p+1} \right \rfloor} \cdot b}{E_{p-1}^{i}} \mod (q^{h’ p}, p^{m’}) : 
					b \in B_i, i = 1,..., \left \lfloor \frac{(p+1)m}{p-1} \right \rfloor\right \}, \] 
		\[S := 
				\left \{ U_p \left(\frac{ p^{\left \lfloor \frac{i}{p+1} \right \rfloor} \cdot b}{E_{p-1}^{i}}\right) \mod (q^{h’}, p^{m’}) : 
					b \in B_i, i = 1,..., \left \lfloor \frac{(p+1)m}{p-1} \right \rfloor\right \}. \] 
	This allows us to essentially work over $\Z_p$ instead of $B$, while still obtaining correct outputs, 
	up to some precision (which can then be verified at the end). 
	Similarly, in Section 2.1.1 of \cite{Lau14}, 
	the author explains that $r$ plays a purely auxiliary role, 
	and that from a computational point of view, one can alternatively let a $p^{\alpha}$-overconvergent modular form, 
	for any $\alpha \in \Q_{>0}$, be a series 
	of the form $\sum_{i=0}^{\infty} p^{\left \lfloor \alpha i \right \rfloor} \frac{b_i}{E_{p-1}^{i}}$ 
	with $b_i \in A_{k+i(p-1)}(\Z_p,N)$, similarly to Proposition \ref{OCMFs3}. 
	This space would be denoted by $M^{\text{oc}}_{k}(\Z_p, N, \chi ; p^{\alpha})$, even when $p^{\alpha} \not \in \Z_p$. 
	Following Lauder's approach, we will assume for our computations that $r$-overconvergent modular forms, 
	for $r \in B$ with $\ord(r) =1/(p+1)$, can be expressed as elements of $M^{\text{oc}}_{k}(\Z_p, N, \chi ; p^{1/(p+1)})$. 
	\end{rmk}
	
	The advantage of the above algorithm is that we only need to compute $U_p$ once on the Katz basis 
	and then we will be able to apply the Atkin operator as many times as we wish without having to actually use its original definition. 
	Given an overconvergent modular form $f$ of 
	growth condition $r$, with $\ord(r)=\frac{p}{p+1}$, 
	we can express it as a sum 
		\begin{equation} \label{Nov110407p} 
			f = \sum_{i} \epsilon_{i}  v_i \mod (q^{h’ p}, p^{m’}). 
		\end{equation} 	
	Write $[f]_{\text{Kb}}:=(\epsilon_1,...,\epsilon_d)$ 
	and compute $A [f]_{\text{Kb}}$. 
	Letting $\gamma_{i}$ denote the entries of $[U_p(f)]_{\text{Kb}}$, 
	we finally obtain  
		\begin{equation} \label{23sep250136a}
			U_p(f) = \sum_{i} \gamma_{i} v_{i} \mod (q^{h’}, p^{m}). 
		\end{equation} 
	Thus, we have $[U_p(f)]_{\text{Kb}} =A[f]_{\text{Kb}}$. 
	For more details on the correctness Equation (\ref{23sep250136a}), 
	see Section 2.2.2 of \cite{Lau14} and the last paragraph of Section 3.2.1 in \cite{Lau11}. 
	\begin{rmk} \label{Nov110411p}
	Note that the valuation of the growth rate of 
	the overconvergent modular form $f$ in Equation (\ref{Nov110407p}) is $\frac{p}{p+1}$ 
	instead of just $\frac{1}{p+1}$. 
	Although we can write an $r$-overconvergent modular form $\psi$, with $\ord(r)=\frac{1}{p+1}$, in the Katz basis, 
	and $A=[U_p]_{\text{Kb}}$ in the same basis, we cannot directly apply $A$ to $[\psi]_{\text{Kb}}$, 
	as explained in Section 2.2.2 of \cite{Lau14}. 
	Indeed, the coefficients in the expansion of $[\psi]_{\text{Kb}}$ 
	are not guaranteed to decay fast enough ($p$-adically) 
	for our calculations to be accurate and for Equation (\ref{23sep250136a}) to hold. 
	This issue is entirely avoided when the valuation of the growth rate of $\psi$ is $\frac{p}{p+1}$. 
	Thus, when dealing with an overconvergent form $\psi$ of growth rate having valuation $\frac{1}{p+1}$, 
	we have to first compute $U_p(\psi)$ directly (without using the matrix representation $A$ of $U_p$) 
	to obtain an overconvergent form with growth rate of valuation $\frac{p}{p+1}$, 
	thus improving its overconvergence and decay properties. 
	After that, we may apply $A$ to $[U_p(\psi)]_{\text{Kb}}$. 
	\end{rmk}

	To compute ordinary projections $e_{\text{ord}} := \lim U_{p^{n!}}$ of overconvergent modular forms, 
	we pick a big enough $R \in \N$ (cf. Algorithm 2.1 in \cite{Lau14}) 
	such that $A^{R}$ represents $e_{\text{ord}}$ to our desired level of precision. 
	Given an overconvergent modular form $f$, of growth condition having valuation $\frac{p}{p+1}$, 
	written as $\sum_{i} \alpha_{i} v_{i}$ modulo $(q^{h’}, p^{n})$, 
	we compute $\gamma:= A^{R} [f]_{\text{Kb}}$ and let $\gamma_{i}$ denote the entries of $\gamma$. 
	Finally, we obtain 
		\[e_{\text{ord}}(f) = \sum_{i} \gamma_{i} v_{i} \mod (q^{h’}, p^{n}). \] 
 
\section{Algorithmic methods} \label{ExplicitAlgorithmicMethods}

\subsection{Ordinary projections of nearly overconvergent modular forms} \label{NOCProjections} 
	For simplicity, let $\text{d}$ denote the Serre operator $q \frac{\text{d}}{\text{d}q}$. 
	Let $g,h$ be two classical modular forms with coefficients in $\Z$, of weights $\ell,m$ respectively, 
	and let $H:=  \d^{-(1+t)}(g^{[p]}) \times h$, for some integer $t$ with $0 \le t \le \min \{\ell,m\} -2$. 
	We wish to compute 
		\[ \X := e_{\mathrm{ord}} (H) = e_{\mathrm{ord}} \left( \d^{-(1+t)}(g^{[p]}) \times h \right). \]  
	By Theorem 5 of \cite{Ser73}, $\d^{-(1+t)}(g^{[p]})$ is a $p$-adic modular form of weight $\l - 2(1+t)$, 
	so $\X$ is a $p$-adic modular form of weight $\l + m - 2t - 2$. 
	The condition $0 \le t \le \min \{\ell,m\} -2$ ensures that $\X, g$ and $h$ are balanced, 
	i.e. the largest weight is {strictly} smaller than the sum of the other two.  
	If we had that $t=\l -2$, the form $H:=  \d^{-(1+t)}(g^{[p]}) \times h$ would have been overconvergent, 
	as discussed in Section 2.3.2 of \cite{Lau14} (in the paragraph preceding Note 2.3). 
	More precisely, when $t=\l -2$, we have $H \in M_{\l + m - 2t - 2}^{\text{oc}}(K,N; r)$, for any $r \in B$ with $\ord(r) < 1/(p+1)$. 
	Furthermore, for our computations, we follow Lauder's approach (justified in 
	Note 2.3 in \emph{op. cit.}) and make the assumption that $H$ is integral 
	and can also be written as an element of $M^{\text{oc}}_{\l + m - 2t - 2}(\Z_p, N; p^{1/(p+1)})$, 
	as in Remark \ref{25jan141123p}. 
	This was not contradicted by our experiments.
	
	Since $H$ is not necessarily overconvergent in general, 
	we cannot directly use the methods introduced in \cite{Lau14} 
	to compute the ordinary projection $e_{\mathrm{ord}} (H)$. 
	However, $H$ is nearly overconvergent (Proposition 2.9 in \cite{DR14}) 
	and Theorem \ref{Nov050333p} tells us that 
		\[e_{\mathrm{ord}} (H) = e_{\mathrm{ord}} \left( \pi_{\text{oc}} (H) \right) 
				= e_{\mathrm{ord}} \left( \pi_{\text{oc}} \left( \d^{-(1+t)}(g^{[p]}) \times h \right) \right), \]
	where $\pi_{\text{oc}}$ is the overconvergent projection operator. 
	Since $ \pi_{\text{oc}} (H) $ is overconvergent, by definition, 
	we can follow the methods described in \cite{Lau14} 
	to compute its ordinary projection, thus obtaining $e_{\mathrm{ord}} \left( \pi_{\text{oc}} (H) \right)=e_{\mathrm{ord}} (H)$.  
	We therefore turn our attention to computing $\pi_{\text{oc}} (H)$. 
	Note that we are not actually interested in taking the overconvergent projection of any nearly overconvergent modular form; 
	we are specifically computing $\pi_{\text{oc}} ( \d^{-(1+m)}(g^{[p]}) \times h )$. 
	We therefore use a trick (see Theorem \ref{18feb0344p}) 
	that specifically applies to our setting. 
	\\ 
	
	Set $G := \d^{1-\l} g^{[p]}$, it is an overconvergent modular form 
	of weight $2-\l$. 
	Let $n=\l-2-t \ge 0$ so that $\d^{-1-t}g^{[p]}=\d^{n} G$ 
	and $\pi_{\text{oc}}  \left( \d^{-1-t}(g^{[p]}) \times h \right) = \pi_{\text{oc}}  ((\d^{n} G ) \times h)$. 
	Consider the Rankin-Cohen bracket 
		\begin{equation} \label{dec191254a}
			[G, h]_n = \sum_{i,j \ge 0, i+j = n} (-1)^{j} {{(2-\l) + n - 1} \choose{j}} {{m + n - 1}\choose{i}} \d^i(G) \d^j(h). 
		\end{equation} 
	Note that the individual terms in this sum are all $p$-adic modular forms of weight $\l + m - 2t - 2$ 
	that are not necessarily overconvergent. 
	However, the entire sum $[G, h]_n$ is overconvergent. 
	It turns out that the Rankin-Cohen bracket is closely related to the overconvergent projection operator. 
	\begin{theorem} \label{18feb0344p}
		Let $\psi_1,\psi_2$ be overconvergent modular forms of weights $\kappa_1$ and $\kappa_2$ respectively, 
		then, for all $j \ge 0$, 
			\[[\psi_1,\psi_2]_j 
				= { {\kappa_1+\kappa_2+ 2j-2} \choose {j}} 
					\pi_{\text{oc}}  ((\text{d}^{j} \psi_1 ) \times \psi_2). \] 
	\end{theorem} 
	This follows from Section 4.4 of \cite{LSZ20} (see also Theorem 1 in \cite{Lan08}). We thus obtain the following Corollary. 

	\begin{cor} \label{18feb0344p2}
		We can relate $[G, h]_n$ and $\pi_{\text{oc}} ((\text{d}^{n} G ) \times h)$ as follows 
			\[[G, h]_n = { {-\l + m + 2n} \choose {n}} \pi_{\text{oc}} ((\text{d}^{n} G ) \times h). \] 
	\end{cor} 
	
	Thus, we can simply compute $[G, h]_n$ using Equation (\ref{dec191254a}) to obtain $ \pi_{\text{oc}} (H)$. 
	
	\begin{rmk}
		Note that we had to pass through $G$ instead of using $g$ directly 
		as we cannot have the subscript $j$ of the Rankin-Cohen bracket $[\cdot,\cdot]_j$ be negative. 
		Moreover, since the modular forms $\X, g, h$ are balanced, ${ {-\l + m + 2n} \choose {n}}$ cannot be zero. 
	\end{rmk}

\subsection{Eigenspace projections} \label{22sep180911p}

	In the previous sections, we have seen how to compute ordinary projections, 
	i.e. projections over the space of overconvergent modular forms of slope zero. 
	We now consider taking more general projections. 
	
	For all $\sigma \in \Q_{\ge 0}$, 
	the $U_p$ equivariant decomposition of $M_k^{\text{oc}}(N)$, described in Equation (\ref{23Jul120422p}), 
	allows us to express any form $H$ as a sum $H= F_{\sigma} + F$, 
	where $F_{\sigma} \in M_k^{\text{oc}}(N)^{\text{slope } \sigma}$ and $F \in X_{\sigma}$. 
	We call $F_{\sigma}$ the projection of $H$ onto the space of slope $\sigma$, 
	or the slope $\sigma$ projection of $H$. 
	Consider now the generalised eigenspace associated to a single eigenvalue $\mu$ 
	such that $\text{ord}_p(\mu)=\sigma$. 
	We will explain how to project modular forms onto such an eigenspace. 
	This method has been used in \cite{DL21} 
	and is based on an insight of David Loeffler (see the last paragraph of Section 6.3 of \cite{LSZ20}). 
	We call such a projection the \emph{eigenspace $\mu$ projection}. 
	This can be seen as a special case of the slope $\sigma$ projection, as these two notions 
	would agree in the case where $U_p$ only has one eigenvalue $\mu$ of valuation $\sigma$.

	Let $A$ denote the matrix computed in Section \ref{OrdProjOCMFs}, 
	representing the $U_p$ operator acting on the Katz basis of $M^{\text{oc}}_{k}(\Z_p, N, \chi ; p^{1/(p+1)})$ 
	(see Remark \ref{25jan141123p}). 
	Let $\mu$ be an eigenvalue for $U_p$ 
	and let $M=M_{\mu}:= A-\mu \text{Id}$. Put $M_\mu$ in \emph{Smith normal form}, 
	i.e. let $P$ and $Q$ be invertible matrices such that 
		\begin{equation*} 
			Q M_\mu P = D = \text{diag} \left(a_1(\mu),...,a_{s-1}(\mu), a_{s}(\mu) \right), 
		\end{equation*} 
	where $a_1(\mu) | a_2(\mu) | ... | a_s(\mu)$. 
	We now remark that $a_s(\mu)$ should be zero, as $\mu$ is an eigenvalue for $U_p$. 
	However, $A$ is only an approximation for $U_p$. 
	More precisely, $A \in M_{d\times d}(\Z/p^{m}\Z)$ 
	is equal to $U_p$ modulo $p^{m}$. 
	And so, $a_s(\mu)$ will only be zero in $\Z/p^{m}\Z$. 
	Moreover, the case $a_{s-1}(\mu)=0$ 
	happens precisely when $\mu$ has algebraic multiplicity (as an eigenvalue 
	of $U_p$) more than one.
	Assume henceforth that we are dealing with an eigenvalue $\mu$ of algebraic multiplicity one, 
	i.e. that the generalised $\mu$-eigenspace is one-dimensional. 
	\\

	Let $f_{\mu}$ be an eigenform lying in the one-dimensional generalised $\mu$-eigenspace. 
	Assume that $f_{\mu}$ is integral 
	and can be written as an element of $M^{\text{oc}}_{k}(\Z_p, N, \chi ; p^{p/(p+1)})$. 
	Let $\pi= \pi({\mu})$ denote the last row of $Q \in M_{d\times d}(\Z/p^{m}\Z)$, i.e. $\pi_{i}=Q_{i,d}$ for $i=1,...,d$. 
	\begin{prop}\label{23Jun010413p}
		The vector $\pi$ is orthogonal to all $p/(p+1)$-overconvergent modular forms (written in the Katz basis $\emph{Kb}$) 
		not in the generalised $\mu$-eigenspace. 
	\end{prop}
	\begin{proof}
	As we are working with $p/(p+1)$-overconvergent modular forms, 
	we will be able to represent the action of $U_p$ on them by the matrix $A$ given in 
	Section \ref{OrdProjOCMFs}. 
	We start with the simplest case. 
	Let $f_{s}$ be an eigenform of $U_p$ with eigenvalue $s$, 
	such that $s \not = \mu$. 
	Then, $M[f_{s}]_{\text{Kb}}= (A-\mu \text{Id})[f_s]_{\text{Kb}}=(s-\mu)[f_s]_{\text{Kb}}$. 
	Hence, 
		\begin{equation} \label{dec070812}
			Q(s-\mu)[f_s]_{\text{Kb}} = QM[f_s]_{\text{Kb}} = DP^{-1} [f_s]_{\text{Kb}}. 
		\end{equation} 
	Since $\pi$ is the last row of $Q$ and the last row of $D$ is completely zero, Equation (\ref{dec070812}) gives 
		\begin{equation} \label{dec070814}
			(s-\mu) \pi [f_s]_{\text{Kb}} = \pi (s-\mu) [f_s]_{\text{Kb}} = 0. 
		\end{equation} 
	As $s \not = \mu$, we must have $\pi [f_s]_{\text{Kb}}=0$, 
	up to a certain level of precision, as is explained in Remark \ref{dec190104p}. 
	This shows that any eigenform of $U_p$, with eigenvalue not $\mu$, is orthogonal to $\pi$. 
	
	A similar argument applies to generalised eigenform. 
	Let $F_s$ be a generalised eigenform for the eigenvalue $s$, again with $s \not = \mu$. 
	There exists some minimal integer $r \in \N$ such that $(A-s \text{Id})^{r} [F_s]_{\text{Kb}}=0$. 
	Let $M_{s}:= A-s \text{Id}$, so that $M_s^{r}[F_s]_{\text{Kb}}=0$. Then, 
		\begin{align*}
			\left( M- M_s \right)^{r}[F_s]_{\text{Kb}} 
					= M \sum_{i=0}^{r-1} {r \choose i} (-1)^{i} M^{r-1-i} M_s^{i} [F_s]_{\text{Kb}}. 
		\end{align*} 
	Therefore, $\left( M- M_s \right)^{r}[F_s]_{\text{Kb}} = M C [F_s]_{\text{Kb}}$, where $C:=\sum_{i=0}^{r-1} {r \choose i} (-1)^{i} M^{r-1-i} M_s^{i}$. 
	Now, $\left( M- M_s \right)^{r}=(s-\mu)^r \text{Id}$, hence 
		\begin{equation} \label{dec070919}
			(s-\mu)^{r} \cdot Q [F_s]_{\text{Kb}} = Q \left( M- M_s \right)^{r}[F_s]_{\text{Kb}} = Q M C [F_s]_{\text{Kb}} = DP^{-1}C [F_s]_{\text{Kb}}. 
		\end{equation} 
	And as above, Equation (\ref{dec070919}) gives 
		\begin{equation} \label{dec070920}
			(s-\mu)^{r} \pi [F_s]_{\text{Kb}} =  0. 
		\end{equation} 
	Finally, since $s \not = \mu$, we have $\pi [F_s]_{\text{Kb}} =0$, 
	up to a certain level of precision (see Remark \ref{dec190104p}). 
	That is, $\pi$ must be orthogonal to all overconvergent modular forms not in the generalised $\mu$-eigenspace. 
	\end{proof} 
		\begin{rmk} \label{dec190104p}
		Following 
		Equations (\ref{dec070814}) and (\ref{dec070920}), 
		we say that $\pi [f_s]_{\text{Kb}}$ and $\pi [F_s]_{\text{Kb}}$ are zero, up to some (potentially lower) precision. 
		Indeed, as we are working over $\Z/p^{m}\Z$ for some $m \in \Z$, 
		Equation (\ref{dec070920}) actually becomes $p^{m} | (s-\mu)^{r} \pi [f_s]_{\text{Kb}}$, 
		which does not necessarily imply that $p^{m} | \pi [f_s]_{\text{Kb}}$. 
		Therefore, there is a loss of precision of $r \cdot \text{ord}_p(s-\mu)$. 
		This specific loss of precision can be bounded above by looking at the largest non-zero entry of $D$, 
		since $\text{ord}_p(s-\mu) \le \max_{i} \text{ord}_p(D_{i,i})$. 
		To see this, using Equation (\ref{dec070812}), write 
			\[(s-\mu) \ \text{row}_i(Q) \cdot [f_s]_{\text{Kb}} = D_{i,i} \ \text{row}_i(P^{-1}) \cdot [f_s]_{\text{Kb}}. \] 
	\end{rmk} 
	
	We now explain how to compute the projection $e_{\text{eigenspace } \mu} (H)$ 
	of an overconvergent modular form $H$, 
	seen as a an element of $M^{\text{oc}}_{k}(\Z_p, N, \chi ; \frac{p}{p+1})$, 
	over the generalised $\mu$-eigenspace. 
	Since we are assuming that the generalised $\mu$-eigenspace is one dimensional, 
	we can also call this projection a projection over $f_{\mu}$ and write $e_{f_{\mu}}(H)$, 
	for any non-trivial modular form $f_{\mu}$ in the generalised $\mu$-eigenspace. 
	First, we have that 
		\begin{equation} \label{Dec030204a}
			H= \rho f_{\mu} + \sum_{s \not = \mu} F_{s}, 
		\end{equation} 
	for some constant $\rho$ and generalized eigenforms $F_{s}$ of eigenvalue $s \not = \mu$. 
	This gives us $\pi \cdot [H] = \rho ( \pi \cdot [f_{\mu}])$, by Proposition \ref{23Jun010413p}. 
	Now, since $\pi$ is non trivial, it cannot be orthogonal to all modular forms, 
	so $\pi \cdot [f_{\mu}]$ cannot also be zero. 
	We can hence define : 
		\begin{equation} \label{dec070936} 
			\lambda_{f_{\mu}}(H) := \rho = \frac{\pi \cdot [H]_{\text{Kb}}}{ \pi \cdot [f_{\mu}]_{\text{Kb}}}. 
		\end{equation} 
	More formally, $\lambda_{f_{\mu}}$ 
	is the unique Hecke-equivariant linear functional that factors through the Hecke eigenspace associated to $f_{\mu}$
	and is normalised to send $f_{\mu}$ to $1$ (cf. Definition 2.7 in \cite{Loe18}). 
	And the projection of $H$ over $f_{\mu}$ is 
	simply $e_{\text{eigenspace } \mu} (H)=e_{f_{\mu}}(H):= \lambda_{f_{\mu}}(H) \ f_{\mu}$. 
	
	As explained in Remark \ref{Nov110411p}, 
	this holds under the assumption that $H$ has growth condition $\frac{p}{p+1}$. 
	In the case where $H$ has growth condition $\frac{1}{p+1}$, 
	we need to first apply the Atkin operator to $H$ to obtain a modular form $U_p(H)$ of growth rate $\frac{p}{p+1}$, 
	as in Remark \ref{Nov110411p}. 
	Indeed, write $H$ as a sum $H= \rho f_{\mu} + \sum_{s \not = \mu} F_{s}$, 
	as in Equation (\ref{Dec030204a}). 
	Then, 
		\begin{align*} 
				U_p(H) = \rho \mu f_{\mu} + \sum_{s \not = \mu} U_p(F_{s}). 
		\end{align*} 
	Since the action of $U_p$ preserves the eigenspaces of $M_k^{\text{oc}}(B,N)$, 
	and we are assuming that overconvergent forms over $B$ 
	can be well approximated by forms over $\Z_p$ (Remark \ref{25jan141123p}), 
	we get that 
	$\pi \cdot [U_p(F_s)]_{\text{Kb}}=0$ for $s \not = \mu$, up to some precision (which can then be verified experimentally). 
	So $\pi \cdot [U_p(H)]_{\text{Kb}} = \rho \mu \  \pi \cdot [f_{\mu}]_{\text{Kb}}$. 
	Finally, 
		\[\lambda_{f_{\mu}}(U_p(H)) = \rho \mu =\mu \ \lambda_{f_{\mu}}(H). \] 
	We thus obtain 
		\begin{equation} \label{Nov081257a}
			\lambda_{f_{\mu}}(H) = \frac{ \pi \cdot [U_p(H)]_{\text{Kb}}}{\mu \  \pi \cdot [f_{\mu}]_{\text{Kb}}}. 
		\end{equation} 
	\begin{rmk}
	In the case where the eigenvalue $\mu$ has multiplicity $r$ greater than one, 
	the generalised eigenspace associated to $\mu$ will contain generalised eigenforms other than the one we are projecting on. 
	The method we presented here will thus not work 
	because the projection $e_{\text{eigenspace }\mu}$ 
	over $\mu$-eigenspace is not equal to $e_{f_{\mu}}$ anymore. 
	In this case, one needs to use the last $r$ rows of $Q$ and the other Hecke operators 
	in order to find a system of equations to solve and obtain $\lambda_{f_{\mu}}(H)$. 
	
	As a simple example, assume that we already have a basis for the generalized $\mu$-eigenspace consisting of 
	normalised Hecke eigenforms $\{\mathfrak{f}_1,..., \mathfrak{f}_r\}$, with $\mathfrak{f}_1 = f_{\mu}$. 
	We then express the eigenspace $\mu$ projection of $H$ as a linear combination $\sum_{j} c_{j} \mathfrak{f}_j$. 
	Using the last $r$ rows $\pi_1,...,\pi_r$ of $Q$, 
	we obtain a system of equations $\pi_i \cdot [H] = \sum_{j} c_{j}\ \pi_i \cdot [\mathfrak{f}_j]$. 
	This can easily be solved in order to find $c_1=\lambda_{f_{\mu}}(H)$. 
	The author has not yet implemented this method. 
	\end{rmk} 
	
\section{Triples of modular forms} \label{NewLfunction}
	To a normalised Hecke eigenform $f$ of weight two and level $N$, 
	one can associate a differential $\omega_f \in H_{\text{dR}}^{1}(X_1(N))$. 
	In general, when $f$ has weight $r+2$ and level $\Gamma_1(N)$, 
	one can associate to it a differential $\omega_f \in \text{Fil}^{r+1}H_{\text{dR}}^{r+1}(\mathcal{E}^{r}/\C_p)$, 
	where $\mathcal{E}$ is the universal generalised elliptic curve fibered over $X_1(N)$, 
	and $\mathcal{E}^{r}$ is the Kuga-Sato variety as in \cite{Sch90}. 
	Note that the $f$-isotypic component of $H_{\text{dR}}^{r+1}(\mathcal{E}^{r}/\C_p)$, 
	denoted $H_{\text{dR}}^{r+1}(\mathcal{E}^{r}/\C_p)_{f}$, is two dimensional. 
	Assume now that $f$ is ordinary at $p$. 
	This implies the existence of a one dimensional subspace (the \emph{unit root subspace}) on which the 
	Frobenius endomorphism acts as multiplication by a $p$-adic unit.
	We can then pick a unique element $\eta_{f}^{\text{u-r}}$ in this unit root subspace to 
	extend $\{\omega_f\}$ to a basis $\{\omega_f,\eta_{f}^{\text{u-r}}\}$ 
	such that $\< \omega_{f^{*}}, \eta_{f}^{\text{u-r}}\>=1$, 
	where $\< \cdot , \cdot \>$ is the 
	Poincaré duality pairing on $H_{\text{dR}}^{r+1}(\mathcal{E}^{r}/\C_p)$ and ${f^{*}}:= {f} \otimes \chi_f^{-1}$ 
	(cf. Section 1 of \cite{DR14} and Section 6.1 of \cite{KLZ20}). 
	\\ 
	
	Let $f,g,h$ be three normalised cuspidal modular eigenforms (for all the Hecke operators) of level $N$, 
	respective weights $k,\l,m$ and respective characters $\chi_f,\chi_g,\chi_h$. 
	Fix a prime $p \ge 5$ and assume that $p \nmid N$, 
	that $\chi_f \chi_g \chi_h=1$ and that the weights $k,\l,m$ are balanced, i.e. 
	the largest one is {strictly} smaller than the sum of the other two. 
	The assumption $p \ge 5$ is purely for simplicity and could potentially be relaxed at the cost of some extra care. 
	Let $\alpha_{f},\beta_{f}$ be the roots of the \emph{Hecke polynomial} 
			$x^{2} - a_p(f) x + p^{k-1} \chi_f(p)$. 
	Assume that the modular forms $f,g$ and $h$ are \emph{ordinary}  
	at $p$, 
	so that $\alpha_{f},\alpha_{g}$ and $\alpha_{h}$ are units. 
	Let $f_\alpha$ and $f_\beta$ be the {$p$-stabilizations} of $f$ given by Equation (\ref{22Sep210442p}). 
	Let 
		\[t:= \frac{\l+m-k-2}{2} \ge 0, \qquad c:= \frac{k+ \l+m-2}{2}. \]
	We may then define the Euler factors: 
	\begin{align} \label{22Oct111155a}
	\begin{split}
	\begin{array}{ll}
			\mathcal{E}(f,g,h) := (1-\beta_f \alpha_g \alpha_h p^{-c}) (1-\beta_f \alpha_g \beta_h p^{-c}) (1-\beta_f \beta_g \alpha_h p^{-c}) (1-\beta_f \beta_g \beta_h p^{-c}); \\ 
			\tilde{\mathcal{E}}(f,g,h) := (1-\alpha_f \alpha_g \alpha_h p^{-c}) (1-\alpha_f \alpha_g \beta_h p^{-c}) (1-\alpha_f \beta_g \alpha_h p^{-c}) (1-\alpha_f \beta_g \beta_h p^{-c}) ; 
	\end{array} 
	\\
	\begin{array}{llllllll} 
			\mathcal{E}_0(f) := 1-\beta_f^{2} \chi_f^{-1}(p) p^{1-k}; & & &  \tilde{\mathcal{E}_0}(f) := 1-\alpha_f^{2} \chi_f^{-1}(p) p^{1-k}; & & & \\ 
			\mathcal{E}_1(f) := 1-\beta_f^{2} \chi_f^{-1}(p) p^{-k}; & & & \tilde{\mathcal{E}_1}(f) := 1-\alpha_f^{2} \chi_f^{-1}(p) p^{-k}. & & &
	\end{array}
	\end{split}
	\end{align}

\subsection{The Garrett-Rankin triple product $p$-adic $L$-function} \label{GarrettRankinTPPLF}
	In Section 2.6 of \cite{DR14}, the authors introduce the Garrett-Rankin triple product $p$-adic $L$-function 
	for Hida families $\mathbf{f}, \mathbf{g}, \mathbf{h}$ 
	interpolating $f_{\alpha}$, $g_{\alpha}$ and $h_{\alpha}$ at the weights $k,\ell$ and $m$. 
	The existence of such families is guaranteed by Hida’s construction in \cite{Hid86}. 
	This triple product $p$-adic $L$-function plays an important role in motivating our own $p$-adic symbol $\SS{f}{g}{h}$,  
	defined in the following section. 
	However, as we only define our symbol for integral weight, 
	we only need to describe the Garrett-Rankin triple product $p$-adic $L$-function when it is evaluated at integral weights. 
	
	Let ${f^{*}}:= {f} \otimes \chi_f^{-1}$ and let $f^{*}_{\alpha}$ be the ordinary $p$-stabilization of $f^{*}$. 
	Assume now that the action of the Hecke algebra 
	on the ordinary subspace in weight $k$ is semisimple 
	(cf. page $222$ in \cite{Hid93}). 
	This is the case for $N$ square-free, since $k \ge 2$ (cf. Section 2.3.1 in \cite{Lau14}). 
	This allow us to define (as in Section \ref{22sep180911p}) the unique 
	Hecke-equivariant linear functional $\lambda_{f^{*}_{\alpha}}$ on $S_k^{\text{oc}}(N)$ 
	that factors through the Hecke eigenspace associated to $f^{*}_{\alpha}$
	and that is normalised to send $f^{*}_{\alpha}$ to $1$. 
	
	\begin{defn}[Lemma 2.19, \cite{DR14}] \label{21Jun160521}
		The Garrett-Rankin triple product $p$-adic $L$-value 
		attached to the triple $({f,g,h})$ is defined as 
		\begin{equation} \label{Nov191250p}
				\L_{p} ({f,g,h}):= \lambda_{f^{*}_{\alpha}} \left(\text{d}^{-1-t}{g}^{[p]}  \times {h} \right ). 
			\end{equation} 
	\end{defn} 
	
	In order to experimentally compute the quantity defined in Equation (\ref{Nov191250p}), 
	the main ingredient is the computation of ordinary projections of  $p$-adic modular forms. 
	In \cite{Lau14}, parts of which have been summarized here in Section \ref{OrdProjOCMFs}, 
	the author explains how to calculate the ordinary projections of overconvergent modular forms, 
	and is thus able to compute Garrett-Rankin triple product $p$-adic $L$-values 
	for balanced weights $(k,\l,m)$ satisfying $k=2 + m - \l$. 
	Indeed, this condition guarantees that $\text{d}^{-1-t}({g}^{[p]}) \times {h}$ will be overconvergent, 
	thus the code and the theory in \cite{Lau14} are enough. 
	
	In general, however, when the weights $(k,\l,m)$ are only balanced, $\text{d}^{-1-t}({g}^{[p]}) \times {h}$ 
	is only \emph{nearly overconvergent}. 
	We therefore need to use the generalizations we introduced in Section \ref{NOCProjections}
	in order to compute ordinary projections of nearly overconvergent modular forms, 
	thus being able to compute Garrett-Rankin triple product $p$-adic $L$-values for any balanced classical weights. 
	\\ 
	
	In Section 3 of \cite{DR14}, the authors construct the generalised Gross-Kudla-Schoen diagonal 
	cycle $\Delta:=\Delta_{k,\l,m}$ 
	for a triple of balanced classical weights $(k,\l,m)$. 
	More precisely, this cycle is an element of the Chow group $\text{CH}^{r+2}(W)_{0}$ 
	where  $W:= \E^{k-2} \times \E^{\l-2} \times \E^{m-2}$ and $r:=({k+\l+m})/{2}-3$. 
	One can check from Definition 3.3 of \cite{DR14} that $\Delta_{k,\l,m}$ indeed has codimension $r+2$. 
	Let  
	\begin{equation*}
		\text{AJ}_p: \text{CH}^{r+2}(W)_{0} \longrightarrow  \text{Fil}^{r+2} \text{H}_{\text{dR}}^{2r+3}(W)^{\vee} 
	\end{equation*} 
	be the $p$-adic Abel-Jacobi map (cf. Section (1.2) of \cite{Nek00} or \cite{Bes00}). 
	Darmon and Rotger then show, in Theorem 3.14 of \cite{DR14}, that 
		\begin{equation} \label{21dec061030a}
			\text{AJ}_p(\Delta)(\eta_f^{\text{u-r}} \otimes \omega_g \otimes \omega_h) 
			= (-1)^{t+1} t! \frac{\mathcal{E}_1(f)}{\mathcal{E}(f,g,h)} \<\eta_{f}^{\text{u-r}}, d^{-1-t} g^{[p]} \times h\>, 
		\end{equation} 
	 	where $t:= \frac{\l+m-k-2}{2}$. 
		In Equation (\ref{21dec061030a}), $\eta_f^{\text{u-r}} \in H_{\text{dR}}^{k-1}(\mathcal{E}^{k-2}/\C_p)_{f}$, 
		$\omega_g \in H_{\text{dR}}^{\l-1}(\mathcal{E}^{\l-2}/\C_p)_{g}$ 
		and $\omega_h \in H_{\text{dR}}^{m-1}(\mathcal{E}^{m-2}/\C_p)_{h}$ 
		are as defined at the start of Section \ref{NewLfunction} 
	and 
	we can view the product $\eta_f^{\text{u-r}} \otimes \omega_g \otimes \omega_h$ 
	as an element of $\text{H}_{\text{dR}}^{2r+3}(W)$ thanks to the Künneth decomposition. 
		
		We now 
		provide an alternative way to express special values of the Garrett-Rankin triple product $p$-adic $L$-function 
		by relating them to the generalised Gross-Kudla-Schoen diagonal cycle as follows.  

		\begin{prop}[Theorem 5.1 in \cite{DR14}] \label{22Sep180840p}
			We have 
				\begin{equation*} 
					\text{\emph{AJ}}_p(\Delta)(\eta_f^{\text{u-r}} \otimes \omega_g \otimes \omega_h) 
				= (-1)^{t} t! \frac{\mathcal{E}_0(f) \mathcal{E}_1(f)}{\mathcal{E}(f,g,h)} \lambda_{f_{\alpha}^{*}}(d^{-1-t} g^{[p]} \times h). 
				\end{equation*} 
		\end{prop} 
		\begin{cor} 
			The Garrett-Rankin triple product $p$-adic $L$-value associated to $({f,g,h})$ can be rewritten as 
			\begin{equation} \label{21Jul290608p}
				\L_{p}({f,g,h})
					= \frac{(-1)^{t}}{t!}\frac{\mathcal{E}(f,g,h)}{\mathcal{E}_0(f)\mathcal{E}_1(f)} 
					\text{\emph{AJ}}_p(\Delta)(\eta_f^{\text{u-r}} \otimes \omega_g \otimes \omega_h). 
			\end{equation} 
		\end{cor}
		Equation (\ref{Nov191250p}) 
		provides us with a compact way to express 
		Garrett-Rankin $p$-adic $L$-values. 
		Equation (\ref{21Jul290608p}), on the other hand, connects them 
		to the Abel Jacobi map 
		and provides us with the right insight in order to define a new natural symbol, based on the 
		Garrett-Rankin triple product $p$-adic $L$-function, which we expect to have nice symmetry properties.

\subsection{A new $p$-adic triple symbol $\SS{f}{g}{h}$} \label{22Oct210251p}
	
	Having considered the quantity $\text{AJ}_p(\Delta)(\eta_f^{\text{u-r}} \otimes \omega_g \otimes \omega_h)$ 
	in the previous section, we now turn our attention to $\text{AJ}_p(\Delta)(\omega_f \otimes \omega_g \otimes \omega_h)$, 
	as we believe that the latter should have nice symmetry properties. 
	We investigate such properties further in the following sections. 
	Let us now provide a way to express $\text{AJ}_p(\Delta)(\omega_f \otimes \omega_g \otimes \omega_h)$ 
	in terms of projections onto isotypic spaces, similarly to Proposition \ref{22Sep180840p}. 
	Assume that the action of the Hecke algebra on the slope $k-1$ subspace is semisimple 
	and define $\lambda_{f_{\beta}^{*}}$ analogously to $\lambda_{f_{\alpha}^{*}}$ 
	at the start of Section \ref{GarrettRankinTPPLF}. 
	Let 
		\begin{equation} \label{22sep290735p}
			\l_{fgh,\alpha} := \lambda_{f_{\alpha}^{*}} \left( \text{d}^{-1-t}(g^{[p]} ) \times {h} \right); \qquad 
			\l_{fgh,\beta} := \lambda_{f_{\beta}^{*}} \left( \pi_{\text{oc}} \left( \text{d}^{-1-t}(g^{[p]} ) \times {h} \right)\right). 
		\end{equation} 
	Note that including $\pi_{\text{oc}}$ before $\lambda_{f_{\alpha}^{*}}$ in (\ref{22sep290735p}) 
	would be redundant, by Theorem \ref{Nov050333p}. 
\begin{lem} \label{Jun22130136p} 
	Let $f$ be a classical eigenform of weight $k$ that is ordinary at $p$ with $\text{{ord}}_p(\alpha_{f})=0$. 
	Then, we have $e_{\text{ord}}(f) = \frac{1}{\mathcal{E}_0(f)} f_{\alpha}$ 
	and $e_{\text{slope } k-1}(f)=\frac{1}{\tilde{\mathcal{E}}_0(f)} f_{\beta}$. 
\end{lem} 
\begin{proof}
	We have by definition $f_\alpha (q):= f(q)- \beta f(q^{p})$ and $f_\beta (q):= f(q)- \alpha f(q^{p})$. 
	So, $\alpha f_\alpha - \alpha f = \beta f_\beta - \beta f$. 
	Hence, $\alpha f_\alpha - \beta f_\beta = (\alpha -\beta) f$. 
	Thus, using the notation from (\ref{22Oct111155a}), we have 
		\[e_{\text{ord}}(f) = \frac{\alpha f_\alpha}{\alpha -\beta}  = \frac{1}{\mathcal{E}_0(f)} f_{\alpha}, \qquad 
		e_{\text{slope } k-1} (f) = \frac{\beta f_\beta}{\beta - \alpha}  = \frac{1}{\tilde{\mathcal{E}}_0(f)} f_{\beta}.\] 
\end{proof} 
		\begin{theorem} \label{22Sep180842p}
			Let $t:= \frac{\l+m-k-2}{2}$. 
			We may rewrite $\text{\emph{AJ}}_p(\Delta)(\omega_f \otimes \omega_g \otimes \omega_h)$ as 
			\begingroup\makeatletter\def\f@size{11.2}\check@mathfonts
			\begin{equation*} 
			\text{\emph{AJ}}_p(\Delta)(\omega_f \otimes \omega_g \otimes \omega_h)
	=(-1)^{t}t!\frac{\left\<\omega_f,\phi(\omega_{f^{*}}) \right\>}{p^{k-1}}\left( \frac{\mathcal{E}_1(f)}{\mathcal{E}(f,g,h)} \ \beta_{f^{*}} \ \l_{fgh,\alpha} 
			+\frac{\tilde{\mathcal{E}}_1(f)}{\tilde{\mathcal{E}}(f,g,h)} \ \alpha_{f^{*}} \ \l_{fgh,\beta}  \right). 
			\end{equation*} 
			\endgroup 
		\end{theorem} 
	\begin{proof} 
		Note that $f^{*}$ is orthogonal to the kernel of $e_{f^{*}}$, 
		so $\<f^{*},\psi\>=\< f^{*}, e_{f^{*}} (\psi)\>$ only depends on the projection $e_{f^{*}}(\psi)$ of $\psi$, 
		for any modular form $\psi$. 
		Adapting this to our notation, we obtain $\<\omega_{f^{*}},\psi\>=\< \omega_{f^{*}}, e_{f^{*}} (\psi)\>$. 
		Furthermore, $e_{f^{*}}(\psi)$ only depends on the overconvergent projection of $\psi$. 
		Indeed, $\psi-\pi_{\text{oc}}(\psi)$ is purely nearly overconvergent 
		(i.e. it has no overconvergent part) 
		and will not lie in the $f^{*}$-isotypic space, as $f^{*}$ is overconvergent. 
		Lemma \ref{Jun22130136p} tells us that $f$ has only two slope components: an ordinary one and one of slope $k-1$. 
		Namely, $f= \frac{1}{\mathcal{E}_0(f)} f_{\alpha} + \frac{1}{\tilde{\mathcal{E}}_0(f)} f_{\beta}$, 
		and thus to project over the $f^{*}$-isotypic space, one needs to project over the components $f_{\alpha}^{*}$ and $f_{\beta}^{*}$. 
			Adapting the proof of Proposition \ref{22Sep180840p} 
			for the case of $\text{AJ}(\Delta)(\omega_f \otimes \omega_g \otimes \omega_h)$, 
			and using the notation $\xi(\omega_g , \omega_h)$ from \cite{DR14} (see Equation (72) on p. 30), 
			we write  
			\begin{align*}
			\text{AJ}_p(\Delta)(\omega_f \otimes \omega_g \otimes \omega_h) 
			&= \left \<\omega_f, \xi(\omega_g , \omega_h ) \right \> \\ 
			&= \left\<\omega_f, e_{f^{*},\text{ord}}(\xi(\omega_g,\omega_h))+e_{f^{*}, \text{slope }k-1}(\xi(\omega_g , \omega_h )) \right\>\\ 
			&= \left\<\omega_f, \frac{(-1)^{t} t! \mathcal{E}_1(f)}{\mathcal{E}(f,g,h)}e_{f^{*}, \text{ord}}(d^{-1-t} g^{[p]} \times h) \right\> \\ 
			& \qquad + \left\<\omega_f, \frac{(-1)^{t} t! \tilde{\mathcal{E}}_1(f)}{\tilde{\mathcal{E}}(f,g,h)}e_{f^{*},\text{slope }k-1}\left( \pi_{\text{oc}} ( d^{-1-t} g^{[p]} \times h) \right) \right\> \\ 
			&= (-1)^{t} t! \frac{\mathcal{E}_1(f)}{\mathcal{E}(f,g,h)} \lambda_{f_{\alpha}^{*}}(d^{-1-t} g^{[p]} \times h) \left\<\omega_f, f_\alpha^{*} \right\> \\ 
			& \qquad + (-1)^{t} t! \frac{\tilde{\mathcal{E}}_1(f)}{\tilde{\mathcal{E}}(f,g,h)}\lambda_{f_{\beta}^{*}}\left( \pi_{\text{oc}} ( d^{-1-t} g^{[p]}\times h) \right) \left\<\omega_f,f_{\beta}^{*} \right\>. 
		\end{align*} 
	As $\left\<\omega_f,\omega_{f^{*}} \right\>=0$, we can write 
		\begin{align*}
			\left\<\omega_f,f_{\alpha}^{*} \right\> 
				= \left\<\omega_f,f^{*} - \beta_{f^{*}} Vf^{*} \right\> 
				= - \beta_{f^{*}} \left\<\omega_f, \omega_{Vf^{*}} \right\> 
				= - \frac{\beta_{f^{*}}}{p^{k-1}} \left\<\omega_f, \phi (\omega_{f^{*}}) \right\>. 
		\end{align*} 
	Similarly, $\left\<\omega_f,f_{\beta}^{*} \right\> = - \frac{\alpha_{f^{*}}}{p^{k-1}} \left\<\omega_f, \phi (\omega_{f^{*}}) \right\>$. 
	This gives the desired result. 
	\end{proof} 
	
	We are now ready to write down our candidate for a symmetric $p$-adic triple symbol. 
	\begin{defn} \label{22sep220144p} 
		Let $f,g$ and $h$ be three normalised cuspidal modular eigenforms of level $N$ and respective weights $k, \l$ and $m$ 
		which are regular at $p$. 
		We define the $p$-adic triple symbol $\SS{f}{g}{h}$ by 
		\begin{equation} \label{22sep190248p}
				\SS{f}{g}{h}
					:= (-1)^{t}t! \frac{\left\<\omega_{f}, \phi(\omega_{f^{*}}) \right\>}{p^{k-1}} 
					\left( \frac{\mathcal{E}_1(f)}{\mathcal{E}(f,g,h)} \ \beta_{f^{*}} \ \l_{f g h,\alpha} 
					+\frac{\tilde{\mathcal{E}}_1(f)}{\tilde{\mathcal{E}}(f,g,h)} \ \alpha_{f^{*}} \ \l_{f g h,\beta}  \right). 
		\end{equation}  
	\end{defn} 
	
	In Definition \ref{22sep220144p}, we do not actually need $g$ and $h$ to be cuspidal nor regular at $p$, 
	as Equation (\ref{22sep190248p}) is still defined when only $f$ is. 
	However, as we are interested in permuting the order of $f$, $g$ and $h$, 
	we often require them all to be cuspidal and ordinary at $p$. 
	Thanks to Theorem \ref{22Sep180842p}, we may reformulate $\SS{f}{g}{h}$ 
	as follows. 
	\begin{cor} \label{22sep190246p} 
		We have $\SS{f}{g}{h}=\text{\emph{AJ}}_p(\Delta_{k,\l,m})(\omega_{f} \otimes \omega_{g} \otimes \omega_{h})$. 
	\end{cor} 
	The right hand side of the equation 
	in Corollary \ref{22sep190246p} appears to be symmetric in the variables $f,g,h$, and thus 
	suggests that $\SS{f}{g}{h}$ is symmetric.

\subsection{Symmetry properties of $\SS{f}{g}{h}$} \label{22Oct250500p} 
	We are interested in the behaviour of the $p$-adic triple symbol $\SS{f}{g}{h}$ 
	as we vary the order of ${f,g}$ and ${h}$. 
	We have the following main result. 
	\begin{theorem} \label{22Oct210207p} 
		Let $f,g,h$ be three cuspidal newforms of weights $k,\l,m$ (and the same level). Then 
		$\SS{f}{g}{h}$ satisfies the cyclic symmetry relation 
			\begin{equation*}
				\SS{f}{g}{h} = (-1)^{k} \SS{g}{h}{f} = (-1)^{m}\SS{h}{f}{g}. 
			\end{equation*} 
		In particular, when the weights are all even, $\SS{f}{g}{h}$ is symmetric when its inputs are cyclically permuted. 
	\end{theorem} 
	\begin{proof} 
		We start with the case of weights $k=\l=m=2$.
		In this case, the diagonal cycle $\Delta_{2,2,2}$ is symmetric, 
		as can easily be seen from Definition 3.1 in \cite{DR14}. 
		Recall that $\omega_{f}\otimes\omega_g \otimes \omega_h$ 
		is given by the Künneth decomposition and is therefore composed of cup products. 
		So by the properties of cup products, we have $\omega_{f}\otimes\omega_g = -\omega_g \otimes \omega_f$ 
		and $\omega_{f}\otimes\omega_h = -\omega_h \otimes \omega_f$. 
		We can thus write $\text{AJ}_p (\Delta_{2,2,2})(\omega_{f}\otimes\omega_g \otimes \omega_h)
				= \text{AJ}_p(\Delta_{2,2,2})(\omega_g \otimes \omega_h \otimes \omega_{f})$. 
		
		For general weights $k,\l,m$, a variation of the above holds. 
		We will first study the action of permuting the first two coordinates of $\SS{f}{g}{h}$, 
		then the action of permuting the second and third coordinates and finally combine them to obtain the desired result. 
		We make our argument explicit using the functoriality properties of the $p$-adic Abel Jacobi map. 
		Let $r_1:=k-2,r_2:=\l-2,r_3:=m-2,r:=(r_1+r_2+r_3)/2$ and 
		let $s$ be the map going from $W:= \E^{r_1} \times \E^{r_2} \times \E^{r_3}$ 
		to $W':= \E^{r_2} \times \E^{r_1} \times \E^{r_3}$ that permutes the first and second terms. 
    		Then $s$ induces permutations on the corresponding Chow groups 
    		and de Rham cohomology groups: 
    		we have a pushforward $s_{*}$ 
    		on $\text{CH}^{r+2}(W)_0$ 
    		and a dual pullback $s^{*,\vee}$ 
    		on $\text{Fil}^{r+2} \text{H}_{\text{dR}}^{2r+3}(W)^{\vee}$. 
		The functoriality properties of the $p$-adic Abel Jacobi map with respect to correspondences 
		(see Propositions 1, 2 \& 4 (iii) in \cite{EZZ84}) 
		give us the commuting diagram 
		\begin{equation*}
			\begin{tikzcd} \label{22Oct210214p}
				\text{CH}^{r+2}(W)_{0} \arrow[r, "\text{AJ}_p"] \arrow[d, "s_{*}" ]
				& \text{Fil}^{r+2} \text{H}_{\text{dR}}^{2r+3}(W)^{\vee} \arrow[d, "s^{*,\vee}" ] \\
				\text{CH}^{r+2}(W')_{0} \arrow[r,  "\text{AJ}_p" ]
				& \text{Fil}^{r+2} \text{H}_{\text{dR}}^{2r+3}(W')^{\vee}. 
			\end{tikzcd}
		\end{equation*}
		
		Thus, $\text{AJ}_p s_{*}=s^{*,\vee} \text{AJ}_p$. 
		Given $Z \in \text{CH}^{r+2}(W)_{0}$ and some $\omega \in \text{Fil}^{r+2} \text{H}_{\text{dR}}^{2r+3}(W')$, 
		we get $ \text{AJ}_p (s_{*}Z)(\omega)= (s^{*,\vee}\text{AJ}_p(Z))(\omega)=\text{AJ}_p(Z)(s^{*}\omega)$. 
		We can now apply this to the generalised Gross-Kudla-Schoen diagonal cycle $\Delta_{k,\l,m}$ 
		and take $\omega:=\omega_g \otimes \omega_{f} \otimes \omega_h$. 
		We see that the action of $s^{*}$ on $\omega$ is given 
		by $s^{*}(\omega_g \otimes \omega_{f} \otimes \omega_h) = (-1)^{(k-1)(\l-1)} (\omega_{f} \otimes \omega_g \otimes \omega_h)$, 
		by the skew symmetry of cup products (which are part of the Künneth decomposition). 
		Furthermore, the action of $s_{*}$ on $\Delta_{k,\l,m}$ is given 
		by $s_{*}\Delta_{k,\l,m}=(-1)^{r + (r_1 r_2)}\Delta_{\l,k,m}$. 
		The proof of this is purely combinatorial: 
		one needs to expand Definition 3.3 of $\Delta_{k,\l,m} \in \text{CH}^{r+2}(W)_{0}$ in \cite{DR14} 
		and permute two subsets of $\{1,...,r\}$ of size $r_1$ and $r_2$ and intersection of size $r-r_3$. 
		Finally, $r + r_1r_2 + (k-1)(\l-1)=(k+\l-m)/2 \mod 2$, 
		therefore we obtain the symmetry formula 
			\[\SS{f}{g}{h} = (-1)^{(k+\l-m)/2} \SS{g}{f}{h}. \]
		Similarly, $\SS{f}{g}{h} = (-1)^{(\l+m-k)/2} \SS{f}{h}{g}$. 
		Combining these two symmetry formulas gives $\SS{f}{g}{h} = (-1)^{k} \SS{g}{h}{f}$. 
	\end{proof}
	\begin{rmk}
		When the forms $f,g,h$ are not all new of the same level, 
		there can be cases where the symbol $\SS{f}{g}{h}$ is automatically zero. 
		For instance, for a prime $\l \not =p$, 
		if $f$ is new of level $\l^{2}$ and $g,h$ are level $\l^{2}$ images of forms which are new at level $\l$, 
		we are then projecting the old form $\text{d}^{-1-t}(g^{[p]} ) \times {h}$ onto a new subspace. We thus obtain $0$. 
		It would be interesting to look further into the case of $f,g,h$ not all being new of the same level. 
	\end{rmk}

\section{Examples} \label{22Oct260529p}
\subsection{Computing Poincaré pairings} \label{23Oct150731p} 
	Using our algorithms from Section \ref{ExplicitAlgorithmicMethods}, 
	we can compute $\l_{fgh,\alpha}$ and $\l_{fgh,\beta}$ appearing in Equation (\ref{22sep190248p}). 
	The only remaining factor in this equation that is non-trivial to calculate  
	is the period $\Omega_f:= \left\<\omega_f, \phi (\omega_{f^{*}}) \right\>$. 
	
	When $f$ is a newform (with rational coefficients) of weight $2$, 
	we can use the following trick to calculate $\Omega_f= \left\<\omega_f, \phi (\omega_{f}) \right\>$, 
	as was done in Section 4 of \cite{DL21}. 
	Let $E$ be the elliptic curve associated to $f$. 
	The differential $\omega_f =f \frac{\d q}{q}$ corresponds to the differential $\omega_E:= \frac{\d x}{y}$ 
	of the elliptic curve $E$. 
	Computing the Poincaré pairing $\left\<\omega_f, \phi (\omega_{f}) \right\>$ now amounts to 
	calculating $\left\<\omega_E, \text{Frob} (\omega_{E}) \right\>$, 
	up to including the modular degree $m_E$ 
	of $E$: $\left\<\omega_f, \phi (\omega_{f}) \right\>= m_{E} \left\<\omega_E, \text{Frob} (\omega_{E}) \right\>$. 
	Let $M$ 
	be the matrix representing the action of Frobenius, up to precision $p^{m}$, 
	on 
	$\omega_E = \frac{\d x}{y}$ and $\eta_E:= x \frac{\d x}{y}$. 
	Then, 
	$\left\<\omega_E, M \omega_{E} \right\> 
		=\left\<\omega_E, M_{11} \omega_{E} + M_{21} \eta_{E} \right\> = M_{21}$ 
	so that the period $\Omega_f$ is simply given by 
		\[\Omega_f=m_{E} M_{21} \mod p^{m}, \] 
	and the matrix $M$ can be efficiently computed via Kedlaya's algorithm (cf. \cite{Ked01}). 
	
	In the case where $f$ has weight $k$ strictly greater than $2$, we cannot use the above trick anymore. 
	Instead, we can exploit the symmetry of $\SS{*}{*}{*}$ and the algorithms mentioned so far in this paper. 
	Indeed, in order to calculate the period $\Omega_f$, 
	we first appropriately chose two auxiliary forms $f_0$ and $\varphi$, 
	such that $\Omega_{f_0}$ is known or computable (e.g. when $f_0$ has weight $2$). 
	Then, using the symmetry relation  of Theorem \ref{22Oct210207p}, 
	we obtain $\SS{f}{f_0}{\varphi} = (-1)^{k} \SS{f_0}{\varphi}{f}$. 
	The right hand side containing $\Omega_{f_0}$ is entirely known, 
	whereas the left hand side is entirely computable except for $\Omega_{f}$. 
	We can thus recover the value of $\Omega_{f}$. 
	
	This method is explained is great detail in Section 6.2 of \cite{GhaThesis23}. 
	It is however simpler to illustrate it by means of examples. 
	See in particular Examples \ref{22Apr191034a}, \ref{22Apr200335p} and \ref{23oct151123p} in the next section. 
	
\subsection{Symmetry relations for even weights} \label{SymRel} 
	We dedicate this section to examples illustrating our symmetry relations (from Theorem \ref{22Oct210207p}) 
	in the case of even weights. 
	This provides an application to the algorithms described in Section \ref{ExplicitAlgorithmicMethods}. 
	\\ 
	
	We begin 
	with a simple case 
	where all modular forms have weight $2$. 
	This only involves overconvergent modular forms, 
	and we can compute $\Omega_f$ via Kedlaya's algorithm. 
	\begin{eg} \label{22Sep210816p}
	Consider the space of newforms $S^{\text{new}}_2(\Q, 57)$ of weight $2$ and level $57$. 
	Let $f,g$ and $h$ be the cuspidal newforms in $S^{\text{new}}_2(\Q, 57)$ given by: 
		\begin{align*}
			f &= q - 2q^2 - q^3 + 2q^4 - 3q^5 + 2q^6 - 5q^7 + q^9 + 6q^{10} + q^{11} + ... , \\ 
			g &= q + q^2 + q^3 - q^4 - 2q^5 + q^6 - 3q^8 + q^9 - 2q^{10} + ... , \\ 
			h &= q - 2q^2 + q^3 + 2q^4 + q^5 - 2q^6 + 3q^7 + q^9 - 2q^{10} - 3q^{11} + ... . 
		\end{align*} 
	Fix $p:=5$ and let $f_{\alpha_{f}}$ and $f_{\beta_{f}}$ 
	denote the $p$-stabilizations of $f$ at $p$. 
	Then $f,g$ and $h$ are ordinary at $p$. 
	Using the algorithms described in Section \ref{ExplicitAlgorithmicMethods}, 
	we compute the quantities $\l_{fgh,\alpha}, \l_{fgh,\beta}, \l_{ghf,\alpha}, \l_{ghf,\beta},\l_{hfg,\alpha}, \l_{hfg,\beta}$ and obtain  
		\begingroup\makeatletter\def\f@size{10}\check@mathfonts
		\begin{align*}
			\l_{fgh,\alpha} &= -3774928826965787816511437758179915984738972855613348870149740387513806 \Mod 5^{100}, \\ 
			\l_{fgh,\beta} &= -1600120463087968696799905890349018972704454279824366881678828640068804 \cdot 5^{-1} \Mod 5^{99}, \\ 
			\l_{ghf,\alpha} &= 3414089135682117556340078214096537672013164967359802729338191598002457 \cdot 5  \Mod 5^{101}, \\ 
			\l_{ghf,\beta} &= 319324687965512071716318643272796126647017637487474169128482176479703 \Mod 5^{100}, \\ 
			\l_{hfg,\alpha} &= 3386642279338565749426053729955310360166771341172640348803607194424548 \cdot 5^{-1} \Mod 5^{99}, \\ 
			\l_{hfg,\beta} &= -1362182692510584292629393424534010351729144263363030199124032659953338 \Mod 5^{100}, \\   
			\l_{fhg,\alpha} &= 3774928826965787816511437758179915984738972855613348870149740387513806 \Mod 5^{100}, \\ 
			\l_{fhg,\beta} &= 880679317526405930264409438811117931242490011004937901328334255303179 \cdot 5^{-1} \Mod 5^{99},  \\  
			\l_{gfh,\alpha} &= -3414089135682117556340078214096537672013164967359802729338191598002457 \cdot 5  \Mod 5^{101}, \\ 
			\l_{gfh,\beta} &= 1316444872164870993756743549790237920953950571465279247158114744839078 \Mod 5^{100},\\ 
			\l_{hgf,\alpha} &= -1808920468896542138602596599389737900820358470954594339263049333096423 \cdot 5^{-1} \Mod 5^{99}, \\ 
			\l_{hgf,\beta} &= -1848796736101022160118506527593042717532675210104737039492910699421662 \Mod 5^{100}. 
		\end{align*} 
		\endgroup
	Note that we indeed have $\l_{fgh,\alpha} = - \l_{fhg,\alpha}$ and $\l_{fgh,\beta} = - \l_{fhg,\beta}$. 
	In order to experimentally verify the symmetry property of Theorem \ref{22Oct210207p}, 
	we will now compute the 
	periods $\Omega_f, \Omega_g, \Omega_h$. 
	Using Kedlaya’s algorithm, 
	we obtain 
		\begingroup\makeatletter\def\f@size{10.4}\check@mathfonts
		\begin{align*}
			\Omega_f 
				&= 29505681199130962626561255838977599356333294679056282865324073514068 \cdot 5^2  \Mod 5^{100}, \\ 
			\Omega_g 
				&= -159133461381175901704339380528584168392746264473700984619726139435577 \cdot 5  \Mod 5^{100}, \\ 
			\Omega_h 
				&= 78414893708965262061304860105818868793779659587029031834898206619639 \cdot 5^2  \Mod 5^{100}. 
		\end{align*} \endgroup 
	Finally, putting everything together we obtain 
		\begingroup\makeatletter\def\f@size{9.5}\check@mathfonts
		\begin{align*}
			\SS{f}{g}{h}
				&= 5871767952506844465150908265973598858284513190743516082327198557652 \cdot 5^2  \Mod 5^{100}, \\ 
			\SS{g}{h}{f}
				&= 94224189337260166671264507577645656581683633922954092616598438792027 \cdot 5^2  \Mod 5^{100}, \\ 
			\SS{h}{f}{g}
				&= 328989194731033279961794928605802838532429869011399338836233448557652 \cdot 5^2  \Mod 5^{100}. 
		\end{align*} \endgroup 
	And we can check that all these values agree modulo $5^{97}$. 
\end{eg} 

	The next example involves nearly overconvergent modular forms, thus utilizing the full power of Section \ref{22sep180911p}. 
	Since we cannot compute the period $\Omega_f$ as in Example \ref{22Sep210816p}, we use the following ratio trick. 
	We introduce an extra modular form $h_2$ and check for 
	\begin{equation} \label{22apr180257p}
		\frac{\SS{f}{g}{h}}
		{\SS{f}{g}{h_2}}
		\stackrel{\text{\textbf{?}}}{=} \frac{\SS{g}{h}{f}}
		{\SS{g}{h_2}{f}}. 
	\end{equation} 
	This has the advantage of bypassing the calculation of the periods $\Omega_f$ and $\Omega_g$, as they appear in both the 
	numerator and the denominator of Equation (\ref{22apr180257p}). 
	\begin{eg} \label{22apr180218p}
	Let $f,g,h,h_2,h_3 \in S_4(\Q,45)$ be the cuspidal newforms given by: 
		\begin{align*}
			f &=  q - q^2 - 7q^4 - 5q^5 - 24q^7 + 15q^8 + 5q^{10} - 52q^{11} ... , \\ 
			g &= q - 3q^2 + q^4 + 5q^5 + 20q^7 + 21q^8 - 15q^{10} + 24q^{11} ... , \\ 
			h &= q + 4q^2 + 8q^4 + 5q^5 + 6q^7 + 20q^{10} - 32q^{11} + ... , \\ 
			h_2 &=q - 5q^2 + 17q^4 + 5q^5 - 30q^7 - 45q^8 - 25q^{10} - 50q^{11}  + ...  , \\ 
			h_3 &=q + 5q^2 + 17q^4 - 5q^5 - 30q^7 + 45q^8 - 25q^{10} + 50q^{11} + ... .
		\end{align*} 
	For $p:=17$, we have $a_{17}(f) \cdot a_{17}(g) \cdot a_{17}(h) \cdot a_{17}(h_2) \cdot a_{17}(h_3) \not =0$. 
	When considering the $p$-adic symbols $\SS{\psi_1}{\psi_2}{\psi_3}$, 
	for $\psi_i \in \{f,g,h,h_2,h_3\}$ distinct and up to permutations, we have ten potential values to compute. 
	Up to precision $30$ (i.e. in $\Z/17^{30}\Z$), seven give us zero. 
	That is, for $\{\psi_1,\psi_2,\psi_3\} \in \{\{f,g,h\},\{f,g,h_3\}, \{f,h,h_3\},\{f,h_2,h_3\},\{g,h,h_2\},\{g,h_2,h_3\},\{h,h_2,h_3\}\}$, 
	we have $\l_{\psi_1\psi_2\psi_3,\alpha}=\l_{\psi_1\psi_2\psi_3,\beta}=0$. 
	The non-zero values are the ones involving $\{f,g,h_2\}$, $\{f,h,h_2\}$ and $\{g,h,h_3\}$. 
	We compute 
		\begin{align*}
			\SS{f}{g}{h_2}/\Omega_f
				&=  -1023342994315815801374020643871 \cdot 17^{2}  \Mod 17^{30}, \\ 
			\SS{f}{h}{h_2} /\Omega_f
				&=   68362151699300710278000063432 \cdot 17^{2} \Mod 17^{30}, \\ 
			\SS{h_2}{f}{g} /\Omega_{h_2}
				&=  -2631698743570631185431705415466  \cdot 17^{2} \Mod 17^{30}, \\ 
			\SS{h_2}{f}{h} /\Omega_{h_2}
				&=  248547247830740599793540647737 \cdot 17^{2} \Mod 17^{30}. 
		\end{align*} 
	Thus, 
		\[\frac{\SS{f}{g}{h_2}}
			{\SS{f}{h}{h_2} } \bigg/
				 \frac{\SS{h_2}{f}{g}}
					{\SS{h_2}{f}{h}} = 1 \Mod 17^{25}.  \] 
	\end{eg} 
	Now that we have seen how to get around the issue of computing the periods, 
	we show that our algorithms, thanks to Theorem \ref{22Oct210207p}, allow us to recover the value of periods $\Omega_f$ 
	for modular forms $f$ of weight greater than 2, as is described in Section \ref{23Oct150731p}. 
	
	\begin{eg}  \label{22Apr191034a}
	Fix again $p:=17$. 
	Let $f,g,h_2,h_3\in S_4(\Q,45)$ be the same as in Example \ref{22apr180218p} 
	and let $f_0 \in S_2(\Q,45)$ be the newform given by $f_0 =  q + q^2 - q^4 - q^5 - 3 q^8 - q^{10} + ...$. 
	We compute 
		\begin{align*}
			\SS{f_0}{f}{h_3} /\Omega_{f_0}
				&= 16513223984800935050336063815246 \cdot 17^3  \Mod 17^{30}, \\ 
			\SS{f}{h_3}{f_0} /\Omega_{f}
				&= 13539421372161396100812664727177 \cdot 17 \Mod 17^{30}, \\ 
			\SS{f_0}{h_2}{g}/\Omega_{f_0}
				&= -3366884595101012754561302551722 \cdot 17^2 \Mod 17^{30},\\ 
			\SS{h_2}{g}{f_0} /\Omega_{h_2}
				&= 93393936291523115360189136554  \Mod 17^{30}. 
		\end{align*} 
	Using Kedlaya's algorithm, we also compute 
	\[\Omega_{f_0}=\left\<\omega_{f_0}, \phi(\omega_{f_0}) \right\>= 73740522216959426358743952636082111  \cdot 17  \Mod 17^{30}. \]
	Thus, we deduce that we must have 
		\begin{align*} 
			\Omega_f  &= 
				\Omega_{f_0} \cdot \frac{\SS{f_0}{f}{h_3}/\Omega_{f_0}}
						{\SS{f}{h_3}{f_0}/\Omega_{f}} 
					= -8862546113964214628352195959100 \cdot 17^{3} \Mod 17^{27}, \\ 
			\Omega_{h_2}  &= 
				\Omega_{f_0} \cdot \frac{\SS{f_0}{h_2}{g}/\Omega_{f_0}} 
						{\SS{h_2}{g}{f_0}/\Omega_{h_2}} 
					= -1728830956772474294735820116226 \cdot 17^{3} \Mod 17^{26}.
		\end{align*} 
	\end{eg} 
	\begin{eg} 
	Thanks to Example \ref{22Apr191034a}, 
	we have computed $\Omega_f$ and $\Omega_{h_2}$. 
	We can thus go back to Example \ref{22apr180218p} and calculate 
		\begin{align*}
			\SS{f}{g}{h_2} 
				&= -239652798828174535366407660241 \cdot 17^{5}  \Mod 17^{30}, \\ 
			\SS{h_2}{f}{g} 
				&= 5530974613520227843573162330816  \cdot 17^{5} \Mod 17^{30}, \\ 
			\SS{f}{h}{h_2}
				&=-853772346178158460670635373010 \cdot 17^{5} \Mod 17^{30}, \\ 
			\SS{h_2}{f}{h} 
				&= -853772346178158460670635373010 \cdot 17^{5} \Mod 17^{30}. 
		\end{align*} 
	And we have $\SS{f}{g}{h_2} = \SS{h_2}{f}{g} \Mod 17^{30}$ 
	and $\SS{f}{h}{h_2} = \SS{h_2}{f}{h} \Mod 17^{30}$.  
	\end{eg}
We conclude this section with two longer examples involving different modular forms of different weights. 
	\begin{eg} \label{22Apr200335p}
		Fix $p=11$ and let $f_0 \in S_2(\Q,21)$ and $f,g,h \in S_6(\Q,21)$ be the cuspidal newforms given by 
		\begin{align*}
			f_0 &= q - q^2 + q^3 - q^4 - 2q^5 - q^6 - q^7 + 3q^8 + q^9 + 2q^{10} + 4q^{11} +  ... , \\ 
			f &= q + q^2 - 9q^3 - 31q^4 - 34q^5 - 9q^6 - 49q^7 - 63q^8 + 81q^9 - 34 q^{10} - 340 q^{11} +  ... , \\ 
			g &= q + 5q^2 + 9q^3 - 7q^4 + 94q^5 + 45q^6 - 49q^7 - 195q^8 + 81q^9 + 470q^{10} + 52 q^{11} + ... , \\ 
			h &= q + 10q^2 + 9q^3 + 68q^4 - 106q^5 + 90q^6 - 49q^7 + 360q^8 + 81 q^9 - 1060 q^{10} + 92 q^{11} + ... . 
		\end{align*} 
	From Kedlaya's algorithm, we have 
		\[ \Omega_{f_0}=412797842384875685536202567431940950593928402977097 \cdot 11 \Mod 11^{50}. \] 
	Consider the triple $(f_0,f,g)$. We can compute 
		\begin{align*}
			\SS{f_0}{f}{g} /\Omega_{f_0}
				&=   -2257599454326142239276759004266889152843755460 \cdot 11^5 \Mod 11^{49}, \\ 
			\SS{f}{g}{f_0} /\Omega_f
				&=   -2816145142524823359002534585019971120441513443  \Mod 11^{44}, \\ 
			\SS{g}{f_0}{f} /\Omega_{g}
				&=  -1202790078682800562850336220378526707376378726  \Mod 11^{44}. 
		\end{align*} 
	This allows us to recover the periods: 
	\begingroup\makeatletter\def\f@size{9.5}\check@mathfonts
		\begin{align} \label{22Oct190350p}
		\begin{split} 
			\displaystyle \Omega_f &= \Omega_{f_0} \cdot \frac{\SS{f_0}{f}{g}/\Omega_{f_0}}
						{\SS{f}{g}{f_0}/\Omega_f} 
					=-2509689183927003985676644860386486830080817519 \cdot 11^{6}  \Mod 11^{50}, \\ 
			\Omega_g &= \Omega_{f_0} \cdot \frac{\SS{f_0}{f}{g}/\Omega_{f_0}}
						{\SS{g}{f_0}{f}/\Omega_g} 
					= 2597224237884861326788056615405141084095558737 \cdot 11^{6}  \Mod 11^{50}. 
		\end{split}
		\end{align}  
		\endgroup
	Consider now the triple $(f_0,f,h)$. We can compute 
		\begin{align*}
			\SS{f_0}{f}{h} /\Omega_{f_0}
				&=  -2847504000645971661684808020815460021295815552 \cdot 11^4 \Mod 11^{50}, \\ 
			\SS{f}{h}{f_0}/\Omega_f
				&=   208861134786059864497993853997286411529878026 \cdot 11^{-1} \Mod 11^{50}, \\ 
			\SS{h}{f_0}{f} /\Omega_{h}
				&=  150562340318535656035117305085357243695039436  \Mod 11^{50}. 
		\end{align*} 
	This allows us to recover the periods: 
		\begingroup\makeatletter\def\f@size{10}\check@mathfonts
		\begin{align} \label{22Apr200306p}
		\begin{split}
			\Omega_f &= \Omega_{f_0} \cdot \frac{\SS{f_0}{f}{h}/\Omega_{f_0}}
						{\SS{f}{h}{f_0}/\Omega_f} 
					=-934214497598799103313636376811725028664923638 \cdot 11^{6}  \Mod 11^{50}, \\ 
			\Omega_h &= \Omega_{f_0} \cdot \frac{\SS{f_0}{f}{h}/\Omega_{f_0}}
						{\SS{h}{f_0}{f}/\Omega_h} 
					= 1135142804419315201548085509390534816579616824 \cdot 11^{5}  \Mod 11^{49}. 
		\end{split}
		\end{align} 
		\endgroup
	Note that we can also check that the two values we obtained for the period $\Omega_f$ 
	in Equations (\ref{22Oct190350p}) and (\ref{22Apr200306p}) agree modulo $11^{46}$. 
	We can also compute  
		\begin{align*} 
			\SS{f}{g}{h} /\Omega_f
				&= 40268985822287576957977484998251829978986804 \cdot 11^2 \Mod 11^{49}, \\ 
			\SS{g}{h}{f} /\Omega_g
				&=  -52341418987674502913103090342525976869279460 \cdot 11^2 \Mod 11^{49}, \\ 
			\SS{h}{f}{g} /\Omega_{h}
				&= -51832911640971887401862589998201231551663284 \cdot 11^3 \Mod 11^{50}. 
		\end{align*} 

	This finally allows us to calculate, using Equations (\ref{22Oct190350p}) and (\ref{22Apr200306p}), the full values: 
		\begin{align*} 
			\SS{f}{g}{h}
				&=  20986917589986718469194287107276286895307311 \cdot 11^8 \Mod 11^{50}, \\ 
			\SS{g}{h}{f} 
				&= -22914560311143954782518388246573725956557586 \cdot 11^8 \Mod 11^{50}, \\ 
			\SS{h}{f}{g} 
				&=  7861733475215692445486373857156179960213682 \cdot 11^8 \Mod 11^{50}. 
		\end{align*} 
	And we can check that all these values agree modulo $11^{48}$. 
	\end{eg}
	\begin{eg} \label{23oct151123p}
		Fix again $p=11$. Let $f_0 \in S_2(\Q,26)$, $f,g,h \in S_4(\Q,26)$ and $f_1,f_2,f_3 \in S_8(\Q,26)$ 
		be the cuspidal newforms given by 
		\begingroup\makeatletter\def\f@size{11}\check@mathfonts
		\begin{align*}
			f_0 &= q - q^2 + q^3 + q^4 - 3q^5 - q^6 - q^7 - q^8 - 2q^9 + 3q^{10} + 6q^{11} +  ... , \\ 
			f_1 &= q + 8q^2 - 27q^3 + 64q^4 - 245q^5 - 216q^6 - 587q^7 + 512q^8 - 1458q^9 -1960q^{10} - 3874 q^{11} +  ... , \\ 
			f_2 &= q + 8q^2 - 87q^3 + 64q^4 + 321q^5 - 696q^6 - 181q^7 + 512q^8 + 5382q^9 + 2568q^{10} + 7782 q^{11} +  ... , \\ 
			f_3 &= q - 8q^2 - 39q^3 + 64q^4 + 385q^5 + 312q^6 - 293q^7 - 512q^8 - 666q^9 - 3080q^{10} - 5402q^{11} +  ... , \\ 
			f &= q + 2q^2 - q^3 + 4q^4 + 17q^5 - 2q^6 - 35q^7 + 8q^8 - 26q^9 + 34q^{10} + 2q^{11} +  ... , \\ 
			g &= q + 2q^2 + 4q^3 + 4q^4 - 18q^5 + 8q^6 + 20q^7 + 8q^8 - 11q^9 - 36q^{10} - 48 q^{11} + ... , \\ 
			h &= q - 2q^2 + 3q^3 + 4q^4 + 11q^5 - 6q^6 + 19q^7 - 8q^8 - 18q^9 - 22q^{10} - 38q^{11} + ... . 
		\end{align*} 
		\endgroup 
	From Kedlaya's algorithm, we have 
		\[ \Omega_{f_0}= 390581636402185053366232716528660201295552925543487 \cdot 11 \Mod 11^{50}. \] 
	Consider the triple $(f_0,f_1,f_2)$. We can compute 
		\begin{align*}
			\SS{f_0}{f_1}{f_2} /\Omega_{f_0}
				&=  -5933660141750195368504774740219722366045619600 \cdot 11^7 \Mod 11^{50}, \\ 
			\SS{f_1}{f_2}{f_0}/\Omega_{f_1}
				&=   14109208854192176214141915814693455702656065 \cdot 11 \Mod 11^{50}, \\ 
			\SS{f_2}{f_0}{f_1} /\Omega_{f_2}
				&=  -7793794748784781599257971674959575446350726 \cdot 11 \Mod 11^{50}. 
		\end{align*} 
	This allows us to recover the periods: 
		\begingroup\makeatletter\def\f@size{10}\check@mathfonts
		\begin{align*} 
			\Omega_{f_1} = \Omega_{f_0} \cdot 
						\frac{\SS{f_0}{f_1}{f_2}/\Omega_{f_0}}
						{\SS{f_1}{f_2}{f_0}/\Omega_{f_1}} 
					=-210270517651766028348415614154362330709392521 \cdot 11^{7} \Mod 11^{50}, \\ 
			\Omega_{f_2} = \Omega_{f_0} \cdot 
						\frac{\SS{f_0}{f_1}{f_2} /\Omega_{f_0}}
						{\SS{f_2}{f_0}{f_1}/\Omega_{f_2}} 
					=-288814942721593214967913348978722878649507578 \cdot 11^{7} \Mod 11^{50}. 
		\end{align*} 
		\endgroup 
	Now in order to recover $\Omega_f,\Omega_g,\Omega_h$, we compute  
		\begin{align*}
			\SS{f}{f_1}{f_3} /\Omega_{f}
				&=  40903568201933522569570898222005174659773400 \cdot 11^{6} \Mod 11^{47}, \\ 
			\SS{f_1}{f_3}{f}/\Omega_{f_1}
				&=   - 1371650302863648283749356335039702487573085 \cdot 11^{2} \Mod 11^{43}, \\ 
			\SS{g}{f_1}{f_3} /\Omega_{g}
				&=  220420945295555475043577140565385460000211280 \cdot 11^5 \Mod 11^{46}, \\  
			\SS{f_1}{f_3}{g} /\Omega_{f_1}
				&=   1458224252254476116040209429849988597407090 \cdot 11^{2} \Mod 11^{43}, \\ 
			\SS{h}{f_2}{f_2} /\Omega_{h}
				&=   -22167932026142135533189834503070255673967600 \cdot 11^{6} \Mod 11^{47}, \\  
			\SS{f_2}{f_2}{h}/\Omega_{f_2}
				&=  - 1179453771945534511715867212869271933099333 \cdot 11^{2} \Mod 11^{43}. 
		\end{align*} 
	This allows us to recover the periods: 
		\begingroup\makeatletter\def\f@size{11}\check@mathfonts
		\begin{align*} 
			\Omega_f &= \Omega_{f_1} \cdot 
						\frac{\SS{f_1}{f_3}{f}/\Omega_{f_1}}
						{\SS{f}{f_1}{f_3}/\Omega_{f}} 
					= -899774887450008918231593851176607448072958 \cdot 11^{3} \Mod 11^{44}, \\ 
			\Omega_g &= \Omega_{f_1} \cdot 
						\frac{\SS{f_1}{f_3}{g}/\Omega_{f_1}}
						{\SS{g}{f_1}{f_3}/\Omega_{g}} 
					= 36578899966340566317653585313947952362533 \cdot 11^{4} \Mod 11^{45}, \\ 
			\Omega_h &= \Omega_{f_2} \cdot 
						\frac{\SS{f_2}{f_2}{h}/\Omega_{f_2}}
						{\SS{h}{f_2}{f_2} /\Omega_{h}}  
					=  -1778956364295561925487995272361714970219339 \cdot 11^{3} \Mod 11^{44}. 
		\end{align*} \endgroup 
	We finally can calculate the full values: 
		\begin{align*} 
			\SS{f}{g}{h}
				&=  479359167857389648779593478353399577891020  \cdot 11^5 \Mod 11^{46}, \\ 
			\SS{g}{h}{f} 
				&=  1399506016598818090453046501872791514634546  \cdot 11^5 \Mod 11^{46}, \\ 
			\SS{h}{f}{g} 
				&=  2095226804671605448791510983070380539977212 \cdot 11^5 \Mod 11^{46}. 
		\end{align*} 
	And we can check that all these values agree modulo $11^{43}$. 
	\end{eg}

\subsection{The negative sign for odd weights} \label{24oct100841p}
	We now highlight the importance of the negative sign appearing in Theorem \ref{22Oct210207p}. 
	Let $f,g,h$ be modular forms of balanced weights $k,\l,m$ such that $k$ is even and $\l,m$ are odd. 
	Theorem \ref{22Oct210207p} tells us that 
	$\SS{f}{g}{h} = \SS{g}{h}{f} = - \SS{h}{f}{g}$, 
	and we see the factor $(-1)$ appearing in the last term. 
	Requiring that $\SS{f}{g}{h} = \SS{g}{h}{f} = \SS{h}{f}{g}$ (thus removing the negative sign) 
	would imply that $\SS{f}{g}{h}$ is trivial when its inputs do not all have even weights. 
	
	{Experimental evidence} shows that this is not actually the case. 
	This shows that we cannot expect the symmetric formula, with no signs, $\SS{f}{g}{h} = \SS{g}{h}{f} = \SS{h}{f}{g}$ to hold
	if the weights are not all even. 
	We present our examples below. 

\begin{eg}
	Let $\chi$ be the Legendre symbol $\left( \frac{\cdot}{11}\right)$. 
	Let $f_0 \in S_2(\Q,\Gamma_0(11))$ and $f \in S_7(\Q,\Gamma_1(11), \chi)$ 
		be the cuspidal newforms given by 
		\begin{align*}
			f_0 &= q - 2 q^2 - q^3 + 2 q^4 + q^5 + 2 q^6 - 2 q^7 - 2 q^9 - 2q^{10} + q^{11} +  ... , \\ 
			f &= q + 10 q^3 + 64 q^4 + 74 q^5 - 629 q^9 - 1331 q^{11} +  ... . 
		\end{align*} 
	Pick $p=23$. We have $a_p(f_0) \cdot a_p(f) \not = 0$.  
	Using our algorithm, we calculate $\SS{f_0}{f}{f}$ to be 
		\[ 31546925362985192479627183464312205821578431521869740322354 \cdot 23^7 \Mod 23^{50}, \] 
		which is non-zero.  
\end{eg}

\begin{eg}
	Let $\chi$ be the Legendre symbol $\left( \frac{\cdot}{11}\right)$. 
	Let $f_0 \in S_2(\Q,\Gamma_0(11))$ and $f \in S_5(\Q,\Gamma_1(11), \chi)$  
	be the cuspidal newforms given by 
		\begin{align*}
			f_0 &= q - 2 q^2 - q^3 + 2 q^4 + q^5 + 2 q^6 - 2 q^7 - 2 q^9 - 2q^{10} + q^{11} +  ... , \\ 
			f &= q + 7 q^3 + 16 q^4 - 49 q^5 - 32 q^9 + 121 q^{11} +  ... . 
		\end{align*} 
	Pick $p=23$. We have $a_p(f_0) \cdot a_p(f) \not = 0$.  
	Using our algorithm, we calculate $\SS{f_0}{f}{f}$ to be 
		\[ 6507713287936999052116951605714489492434730289541301877894764 \cdot 23^5 \Mod 23^{50},\]
		which is non-zero.  
\end{eg}

\begin{eg}
	Let $f \in S_2(\Q,\Gamma_0(15))$, $g,h \in S_3(\Q,\Gamma_1(15))$  
	be the cuspidal newforms given by 
		\begin{align*}
			f &= q - q^2 - q^3 - q^4 + q^5 + q^6 + 3q^8 + q^9 - q^{10} - 4 q^{11} +  ... , \\ 
			g &= q + q^2 - 3q^3 - 3q^4 + 5q^5 - 3q^6 - 7q^8 + 9q^9 + 5 q^{10} +  ... , \\ 
			h &= q - q^2 + 3q^3 - 3q^4 - 5q^5 - 3q^6 + 7q^8 + 9q^9 + 5 q^{10} +  ... . 
		\end{align*} 
	Pick $p=13$. 
	Note that we actually have $a_p(f) \not = 0$ but $a_p(g) =a_p(h) = 0$ here. 
	This doesn't pose any issues to our algorithms. 
	We obtain 
		\begingroup\makeatletter\def\f@size{10.4}\check@mathfonts
		\begin{align*} 
			\SS{f}{g}{h} = \SS{f}{h}{g} 
				= 57640757896634901611871044405230131156356129425185649 \cdot 13 \Mod 13^{48}
				\not = 0. 
		\end{align*} \endgroup 
	We also see that $\SS{f}{g}{h}$ is symmetric in the $2$nd and $3$rd variables. 
\end{eg}

\begin{eg}
	Let $f \in S_2(\Q,\Gamma_0(15))$, $g,h \in S_5(\Q,\Gamma_1(15))$  
	be the cuspidal newforms given by 
		\begin{align*}
			f &= q - q^2 - q^3 - q^4 + q^5 + q^6 + 3q^8 + q^9 - q^{10} - 4 q^{11} +  ... , \\ 
			g &= q + 7q^2 - 9q^3 + 33q^4 - 25q^5 - 63q^6 + 119q^8 + 81q^9 - 175 q^{10} +  ... , \\ 
			h &= q - 7q^2 + 9q^3 + 33q^4 + 25q^5 - 63q^6 - 119q^8 + 81q^9 - 175 q^{10} +  ... . 
		\end{align*} 
	Pick $p=17$. We have $a_p(f) \cdot a_p(g) \cdot a_p(h) \not = 0$. 
	We then calculate: 
		\begingroup\makeatletter\def\f@size{10}\check@mathfonts
		\begin{align*} 
			\SS{f}{g}{h} = \SS{f}{h}{g} 
				=8960308425349268584612725752076582316781113083897858380 \cdot 17^{5} \Mod 17^{50}
				\not = 0. 
		\end{align*} \endgroup 
\end{eg}
	All the above calculations 
	were performed 
	on a machine with two Intel(R) Xeon(R) E5-2680 CPUs (2.70 GHz) and 256 GB of RAM. 
	The examples in Section \ref{SymRel} took a few hours each. 
	For instance, Example \ref{22Sep210816p}, which had the highest level of precision (namely $\text{mod } 5^{100}$), 
	took around 8 hours and 7 minutes. 
	The examples in Section \ref{24oct100841p} took longer, 
	as these examples dealt with the modular subgroup $\Gamma_1(N)$ instead of $\Gamma_0(N)$. 
	Each of them took roughly a day to compute.

\phantomsection 

\bibliographystyle{alpha} 
\bibliography{References.bib}

\begin{thebibliography}{DLR16}

\bibitem[AI21]{AI21}
Fabrizio Andreatta and Adrian Iovita.
\newblock Triple product $p$-adic ${L}$-functions associated to finite slope
  $p$-adic families of modular forms.
\newblock {\em Duke Mathematical Journal}, 170(9):1989--2083, 2021.

\bibitem[Bes00]{Bes00}
Amnon Besser.
\newblock A generalization of {C}oleman’s $p$-adic integration theory.
\newblock {\em Inventiones mathematicae}, 142:397--434, 2000.

\bibitem[CE98]{CE98}
Robert~F Coleman and Bas Edixhoven.
\newblock On the semi-simplicity of the ${U}_p$-operator on modular forms.
\newblock {\em Math. Ann}, 310:119--127, 1998.

\bibitem[CGJ95]{CGJ95}
Robert~F Coleman, Fernando~Q Gouv{\^e}a, and Naomi Jochnowitz.
\newblock {$E_2$}, $\theta$, and overconvergence.
\newblock {\em International Mathematics Research Notices}, 1995(1):23--41,
  1995.

\bibitem[Col95]{Col95}
Robert~F Coleman.
\newblock Classical and overconvergent modular forms.
\newblock {\em Journal de th{\'e}orie des nombres de Bordeaux}, 7(1):333--365,
  1995.

\bibitem[DL21]{DL21}
Henri Darmon and Alan Lauder.
\newblock Stark points on elliptic curves via {P}errin--{R}iou’s philosophy.
\newblock {\em Annales math{\'e}matiques du Qu{\'e}bec}, pages 1--18, 2021.

\bibitem[DLR16]{DLR16}
Henri Darmon, Alan Lauder, and Victor Rotger.
\newblock Gross--{S}tark units and $p$-adic iterated integrals attached to
  modular forms of weight one.
\newblock {\em Annales math{\'e}matiques du Qu{\'e}bec}, 40:325--354, 2016.

\bibitem[DR14]{DR14}
Henri Darmon and Victor Rotger.
\newblock Diagonal cycles and {Euler} systems {I}: A $p$-adic {Gross-Zagier}
  formula.
\newblock {\em Ann. Sci. {\'E}c. Norm. Sup{\'e}r.(4)}, 47(4):779--832, 2014.

\bibitem[EZZ82]{EZZ84}
Fouad El~Zein and Steven Zucker.
\newblock Extendability of normal functions associated to algebraic cycles.
\newblock {\em Topics in transcendental algebraic geometry, Ann. Math. Stud},
  106:269--288, 1982.

\bibitem[Gha23]{GhaThesis23}
Wissam Ghantous.
\newblock {\em Computational aspects of modular forms and a $p$-adic triple
  symbol}.
\newblock PhD thesis, 2023.

\bibitem[Gou88]{Gou88}
Fernando~Quadros Gouv{\^e}a.
\newblock Arithmetic of $p$-adic modular forms.
\newblock 1304, 1988.

\bibitem[Hid86]{Hid86}
Haruzo Hida.
\newblock Iwasawa modules attached to congruences of cusp forms.
\newblock In {\em Annales scientifiques de l'{\'E}cole Normale Sup{\'e}rieure},
  volume~19, pages 231--273, 1986.

\bibitem[Hid93]{Hid93}
Haruzo Hida.
\newblock {\em Elementary theory of ${L}$-functions and Eisenstein series}.
\newblock Number~26. Cambridge University Press, 1993.

\bibitem[Kat73]{Kat73}
Nicholas~M Katz.
\newblock $p$-adic properties of modular schemes and modular forms.
\newblock In {\em Modular functions of one variable III}, pages 69--190.
  Springer, 1973.

\bibitem[Ked01]{Ked01}
Kiran~S Kedlaya.
\newblock Counting points on hyperelliptic curves using {M}onsky--{W}ashnitzer
  cohomology.
\newblock {\em arXiv preprint math/0105031}, 2001.

\bibitem[KLZ20]{KLZ20}
Guido Kings, David Loeffler, and Sarah~Livia Zerbes.
\newblock Rankin--{E}isenstein classes for modular forms.
\newblock {\em American Journal of Mathematics}, 142(1):79--138, 2020.

\bibitem[Lan08]{Lan08}
Dominic Lanphier.
\newblock Combinatorics of {M}aass--{S}himura operators.
\newblock {\em Journal of Number Theory}, 128(8):2467--2487, 2008.

\bibitem[Lau11]{Lau11}
Alan~GB Lauder.
\newblock Computations with classical and $p$-adic modular forms.
\newblock {\em LMS Journal of Computation and Mathematics}, 14:214--231, 2011.

\bibitem[Lau14]{Lau14}
Alan Lauder.
\newblock Efficient computation of {Rankin} $p$-adic ${L}$-functions.
\newblock In {\em Computations with modular forms}, pages 181--200. Springer,
  2014.

\bibitem[Loe18]{Loe18}
David Loeffler.
\newblock A note on $p$-adic {R}ankin--{S}elberg ${L}$-functions.
\newblock {\em Canadian Mathematical Bulletin}, 61(3):608--621, 2018.

\bibitem[LSZ20]{LSZ20}
David Loeffler, Christopher Skinner, and Sarah~Livia Zerbes.
\newblock Syntomic regulators of {Asai--Flach} classes.
\newblock In {\em Development of Iwasawa Theory—the Centennial of K.
  Iwasawa's Birth}, pages 595--638. Mathematical Society of Japan, 2020.

\bibitem[Nek00]{Nek00}
Jan Nekov{\'a}r.
\newblock $p$-adic {A}bel--{J}acobi maps and $p$-adic heights.
\newblock {\em The Arithmetic and Geometry of Algebraic Cycles (Banff, Canada,
  1998), CRM Proc. Lect. Notes}, 24:367--379, 2000.

\bibitem[Nik11]{Nik11}
Maximilian Niklas.
\newblock {\em Rigid syntomic regulators and the $p$-adic $L$-function of a
  modular form}.
\newblock PhD thesis, 2011.

\bibitem[Sch90]{Sch90}
Anthony~J Scholl.
\newblock Motives for modular forms.
\newblock {\em Inventiones mathematicae}, 100(1):419--430, 1990.

\bibitem[Ser62]{Ser62}
Jean-Pierre Serre.
\newblock Endomorphismes compl{\`e}tement continus des espaces de banach $ p
  $-adiques.
\newblock {\em Publications Math{\'e}matiques de l'IH{\'E}S}, 12:69--85, 1962.

\bibitem[Ser73]{Ser73}
Jean-Pierre Serre.
\newblock Formes modulaires et fonctions z{\^e}ta $p$-adiques.
\newblock In {\em Modular Functions of One Variable III: Proceedings
  International Summer School University of Antwerp, RUCA July 17--August 3,
  1972}, pages 191--268. Springer, 1973.

\bibitem[Urb14]{Urb14}
Eric Urban.
\newblock Nearly overconvergent modular forms.
\newblock In {\em Iwasawa theory 2012}, pages 401--441. Springer, 2014.

\bibitem[Wan98]{Wan98}
Daqing Wan.
\newblock Dimension variation of classical and $p$-adic modular forms.
\newblock {\em Inventiones mathematicae}, 133(2):449--463, 1998.

\end{thebibliography}

$ $

\textit{E-mail address:} \url{wissam.ghantous@ucf.edu}
\end{document}